\documentclass[english,final,3p,times]{elsarticle}
\usepackage[T1]{fontenc}
\usepackage[latin9]{inputenc}
\usepackage{verbatim}
\usepackage{amsmath}
\usepackage{graphicx}
\usepackage{amssymb}

\usepackage{babel}

\usepackage{subfig}
\usepackage{multirow}

\makeatletter

\usepackage{amsthm}
\usepackage{epsfig}
\usepackage{psfrag}
\usepackage{float}
\usepackage{graphics}
\usepackage{graphicx}
\usepackage{subfig}
\usepackage{booktabs}
\usepackage{pst-all}
\usepackage{pst-poly}
\usepackage{multido}
\DeclareGraphicsExtensions{.eps,.ps,.pdf,.png}
\usepackage{algorithm}
\usepackage{algorithmic}
\usepackage{bm}

\psset{unit=1pt,fillcolor=white}

\newcommand{\Reals}{\mathbb{R}}
\newcommand{\Int}{\mathbb{Z}}

\newcommand{\rme}{\mathrm{e}}
\newcommand{\rmi}{\mathrm{i}}
\newcommand{\rmF}{\mathrm{F}}

\newcommand{\abs}[1]{|{#1}|}
\newcommand{\bigabs}[1]{\bigl|{#1}\bigr|}

\newcommand{\Bigparen}[1]{\Bigl({#1}\Bigr)}

\newcommand{\bigbracket}[1]{\bigl[{#1}\bigr]}

\newcommand{\set}[1]{\{{#1}\}}

\newcommand{\Bigset}[1]{\Bigl\{{#1}\Bigr\}}

\newcommand{\ip}[2]{\langle{#1},{#2}\rangle}
\newcommand{\bigip}[2]{\bigl\langle{#1},{#2}\bigr\rangle}

\DeclareMathOperator*{\argmax}{argmax}
\DeclareMathOperator*{\ssum}{sum}

\newtheorem{theorem}{Theorem}[section]

\newtheorem{corollary}[theorem]{Corollary}

\@ifundefined{showcaptionsetup}{}{%
 \PassOptionsToPackage{caption=false}{subfig}}
\usepackage{subfig}
\makeatother

\usepackage{babel}

\begin{document}

\title{Guaranteeing Convergence of Iterative Skewed Voting Algorithms\\ for Image Segmentation}

\author[db]{Doru C. Balcan\corref{cor1}}
\ead{dbalcan@cc.gatech.edu}

\author[gs]{Gowri Srinivasa}
\ead{gsrinivasa@pes.edu}

\author[mf]{Matthew Fickus}
\ead{Matthew.Fickus@afit.edu}

\author[jk]{Jelena Kova\v{c}evi\'{c}}
\ead{jelenak@cmu.edu}

\address[db]{School of Interactive Computing, Georgia Institute of Technology, Atlanta, USA}
\address[gs]{Dept. of Information Science and Engineering, and Center for Pattern Recognition, PES School of Engineering, Bangalore, India}
\address[mf]{Dept. of Mathematics and Statistics, Air Force Institute of Technology, Wright-Patterson AFB, USA}
\address[jk]{Dept. of Biomedical Eng., Electrical and Computer Eng. and Center for Bioimage Informatics, Carnegie Mellon University, Pittsburgh, USA}

\cortext[cor1]{Corresponding author}
\fntext[fn1]{Email: dbalcan@cc.gatech.edu. Mailing address: College of Computing Building, room 218, Georgia Institute of Technology, 801 Atlantic Drive, Atlanta, GA 30332-0280. Phone: (404) 385-8547, Fax: (404)894-0673.}

\begin{abstract}
In this paper we provide rigorous proof for the convergence of an iterative voting-based image segmentation algorithm called Active Masks. Active Masks (AM) was proposed to solve the challenging task of delineating punctate patterns of cells from fluorescence microscope images. Each iteration of AM consists of a linear convolution composed with a nonlinear thresholding; what makes this process special in our case is the presence of additive terms whose role is to "skew" the voting when prior information is available.  In real-world implementation, the AM algorithm always converges to a fixed point. We study the behavior of AM rigorously and present a proof of this convergence. The key idea is to formulate AM as a generalized (parallel) majority cellular automaton, adapting proof techniques from discrete dynamical systems.
\end{abstract}
\maketitle

\section{Introduction}
\label{sec:intro}

Recently, a new algorithm called \textit{Active Masks} (AM) was proposed for the segmentation of biological images~\cite{Srinivasa:09}.  Let the ``image'' $f$ be any real-valued function over the domain $\Omega:=\prod_{d=1}^D\Int_{N_d}$ and refer to the $N:=N_1N_2{\ldots}N_D$ elements of $\Omega$ as \textit{pixels}; here, $\Int_{N_d}$ denotes the finite group of integers modulo $N_d$.  A \textit{segmentation} of $f$ assigns one of $M$ possible \textit{labels} to each of the $N$ pixels in $\Omega$.  For the fluorescence microscope image depicted in Figure~\ref{fig:AM}(a), one example of a successful segmentation is to label all of the background pixels as ``$1$,'' assign label ``$2$'' to every pixel in the largest cell, ``$3$'' to every pixel in the second largest cell, and so on.  Formally, a segmentation is a \textit{label function} $\psi:\Omega\rightarrow\{1,2,\dots,M\}$, or, equivalently,  a collection of $M$ binary \textit{masks} $\mu_m:\Omega\rightarrow\{0,1\}$ where, at any given $n\in\Omega$, we have $\mu_{m}(n)=1$ if and only if $\psi(n)=m$.  That is, $\mu_m$ at any iteration $i$ can be defined as 
\begin{equation*}
\label{eq:mu}
\mu_m^{(i)}:=\left\{\begin{array}{ll}1,&\psi_{i}(n)=m,\\0,&\psi_{i}(n)\neq m,\end{array}\right.
\end{equation*}
In AM, these masks actively evolve according to a given rule. To understand this evolution, it helps to first discuss \emph{iterative voting}: in each iteration, at any given pixel, one counts how often a given label appears in the neighborhood of that pixel---weighting nearby neighbors more than distant ones---and assigns the most frequent label to that pixel in the next iteration.  For example, if a pixel labeled ``$1$'' in the current iteration is completely surrounded by pixels labeled ``$2$'', its label will likely change to ``$2$'' in the next iteration.  Formally speaking, iterative voting is the repeated application of the rule:
\begin{equation}
\label{eq:ct}
\text{Iterative Voting:}
\qquad\psi_{i}(n)=\argmax\limits_{1\leq m\leq M}~\bigbracket{(\mu_{m}^{(i-1)}*g)(n)},
\end{equation}
where $i$ is the index of the iteration, $g:\Omega\rightarrow\mathbb{R}$ is some arbitrarily chosen fixed weighting function and ``$*$'' denotes circular convolution over $\Omega$.  Iterative voting is referred to as a \textit{convolution-threshold} scheme since it simplifies to rounding the filtered version of $\mu_1^{(i)}$ in the special case $M=2$.  Experimentation reveals that for typical low-pass filters $g$, repeatedly applying~\eqref{eq:ct} to a given initial $\psi_{0}$ results in a progressive smoothing of the contours between distinctly labeled regions of $\Omega$.  Despite this nice property, note that taken by itself, iterative voting is useless as a segmentation scheme, as~\eqref{eq:ct} evolves masks in a manner that is independent of any image under consideration.

The AM algorithm is a generalization of~\eqref{eq:ct} that contains additional image-based terms whose purpose is to drive the iteration towards a meaningful segmentation.  To be precise, the AM iteration is:
\begin{equation}
\label{eq:am}
\text{Active Masks:}
\qquad\psi_{i}(n)=\argmax\limits_{1\leq m\leq M}~\bigbracket{(\mu_{m}^{(i-1)}*g)(n)+R_{m}(n)},
\end{equation}
where the region-based distributing functions
 $\{R_m\}_{m=1}^M$ can be any image-dependent real-valued functions over $\Omega$.  
These will be referred to as {\em skew functions} in this paper, due to their role to bias the voting.
 Essentially, at any given pixel $n$, these additional terms skew the voting towards labels $m$ whose $R_m(n)$ values are large. For good segmentation, one should define the $R_m$'s in terms of features in the image that distinguish regions of interest from each other.

For example, for the fluorescence microscope image given in Figure~\ref{fig:AM}(a), the cells appear noticeably brighter than the background. As such, we choose $R_1$ to be a soft-thresholded version of the image's local average brightness, and choose the remaining $R_m$'s to be identically zero.  When \eqref{eq:am} is applied, such a choice in $R_m$'s forces pixels which lie outside the cells towards label ``$1$,'' while pixels that lie inside a cell can assume any other label.  Intuitively, repeated applications of~\eqref{eq:am} will cause the mask \smash{$\mu_1^{(i)}$} to converge to an indicator function of the background, while each of the other masks $\{\mu_m^{(i)}\}_{m=2}^{M}$ converges either to a smooth blob contained within the foreground or to the empty set.  Experimentation reveals that the AM algorithm indeed often converges to a $\psi$ which assigns a unique label to each cell provided the scale of the window $g$ is chosen appropriately~\cite{Srinivasa:09}; see Figure~\ref{fig:AM} for examples.
\begin{figure*}
\centering
\subfloat[Original image]{
\includegraphics[width=0.2\textwidth]{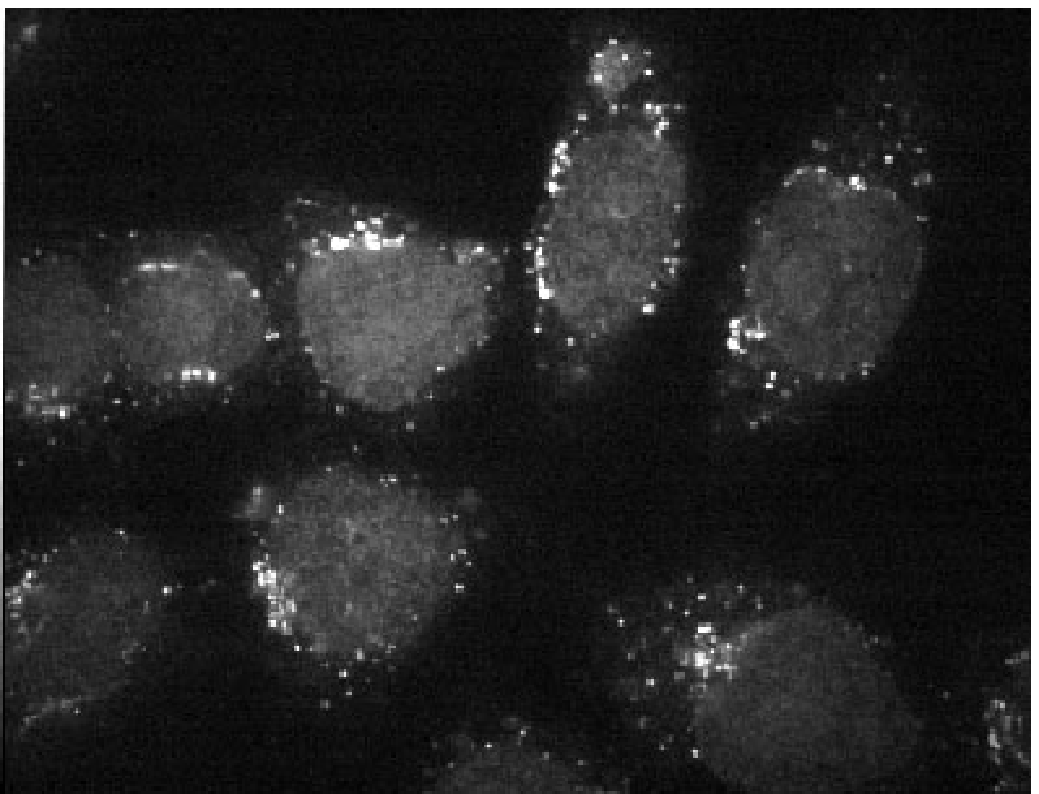}}
\qquad
\subfloat[$i$=0, M=256]{
\includegraphics[width=0.2\textwidth]{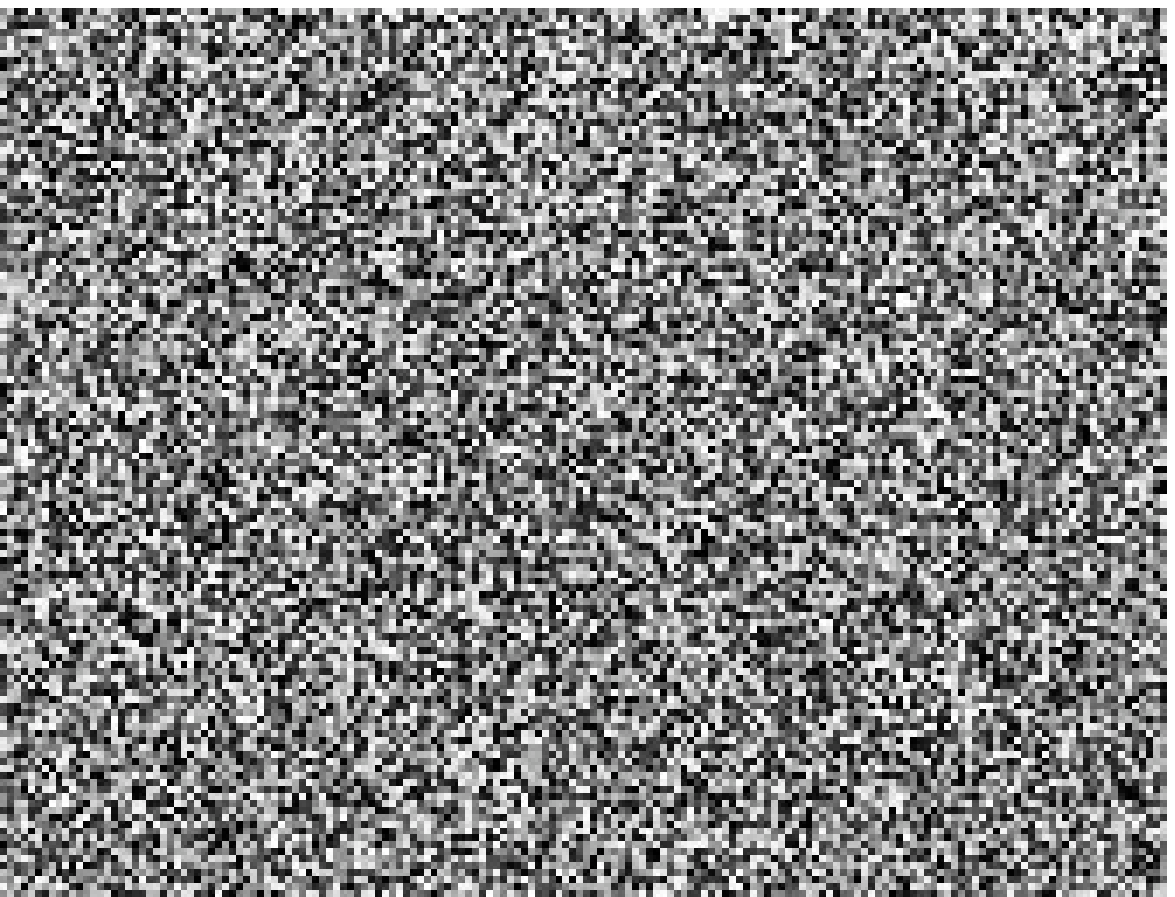}}\\
\subfloat[Segmentation outcomes for various scales]{
\centering
\begin{tabular}{rccc}
& scale$=4$
& scale$=16$
& scale$=32$\\
{Gaussian Filter}
&\includegraphics[height=0.05\textwidth,width=0.2\textwidth]{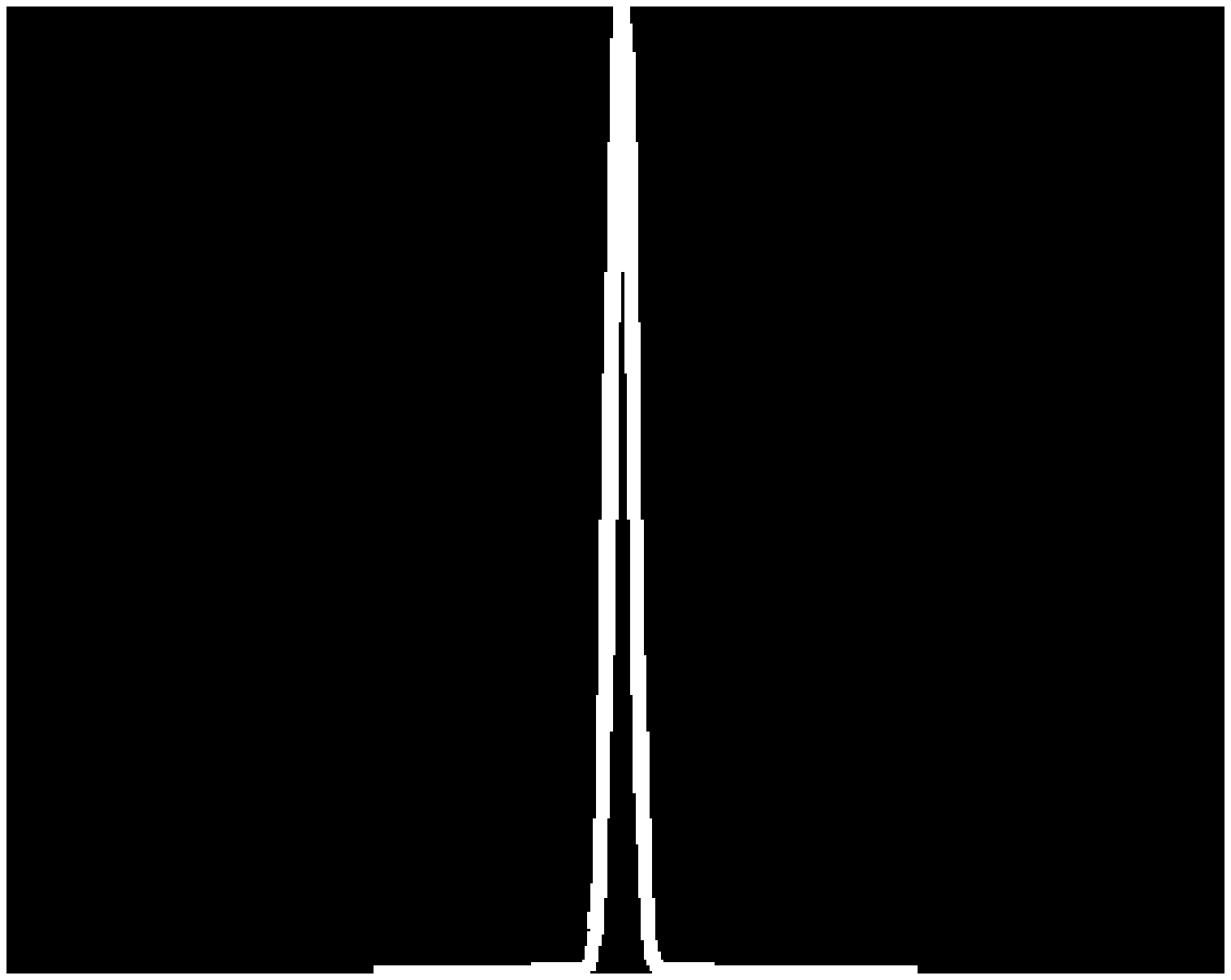}
&\includegraphics[height=0.05\textwidth,width=0.2\textwidth]{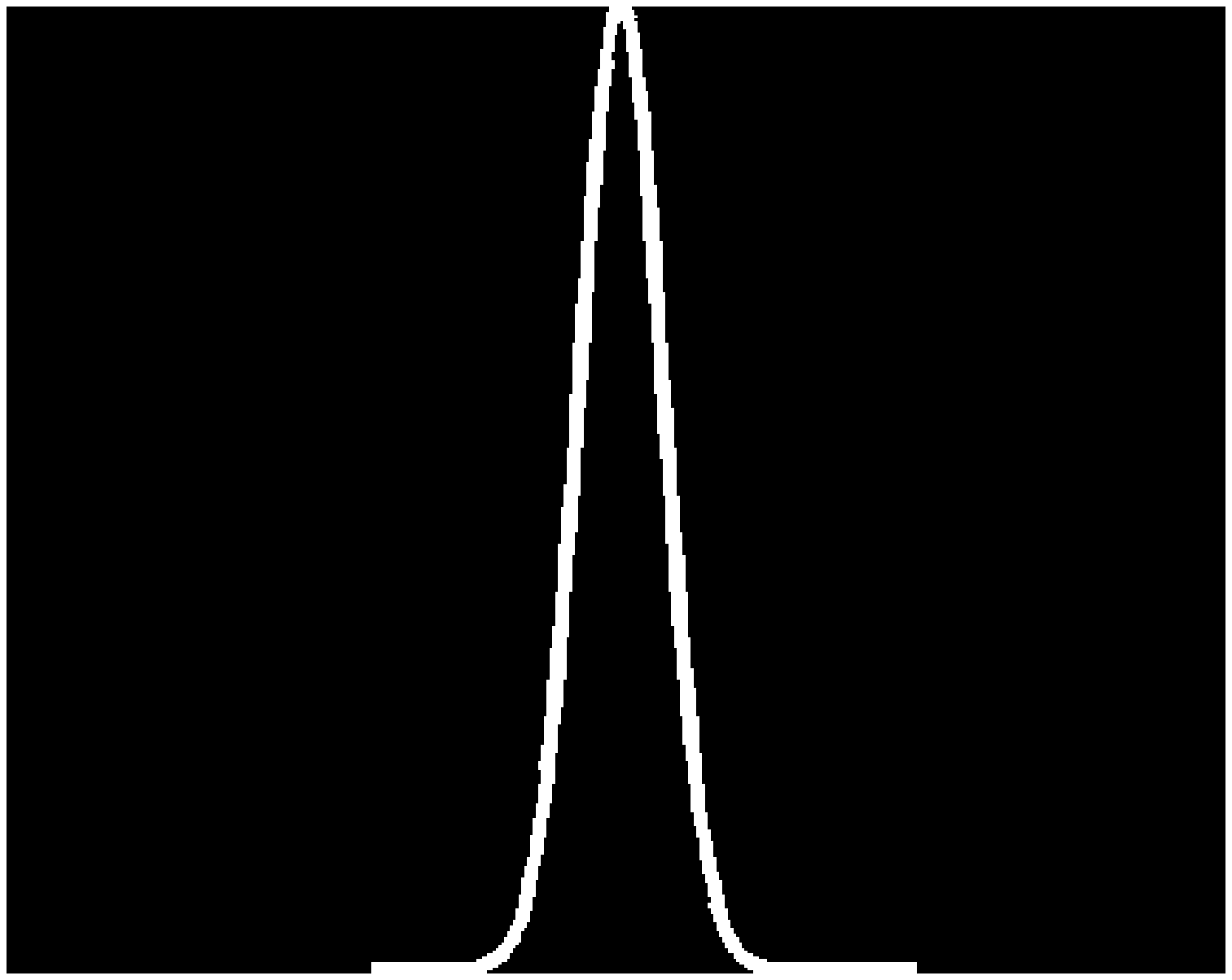}
&\includegraphics[height=0.05\textwidth,width=0.2\textwidth]{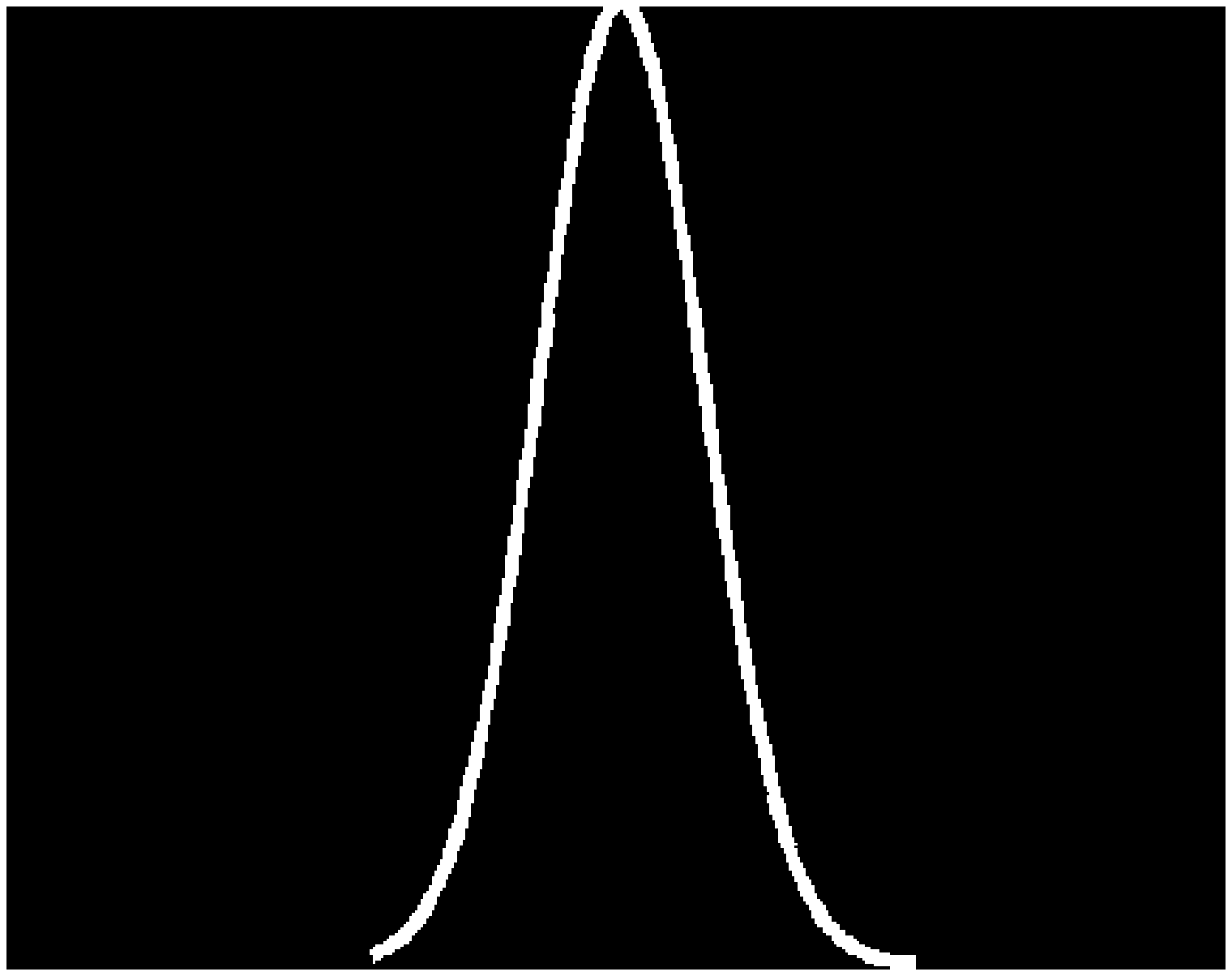}\\
{$i=2$}
&\includegraphics[width=0.2\textwidth]{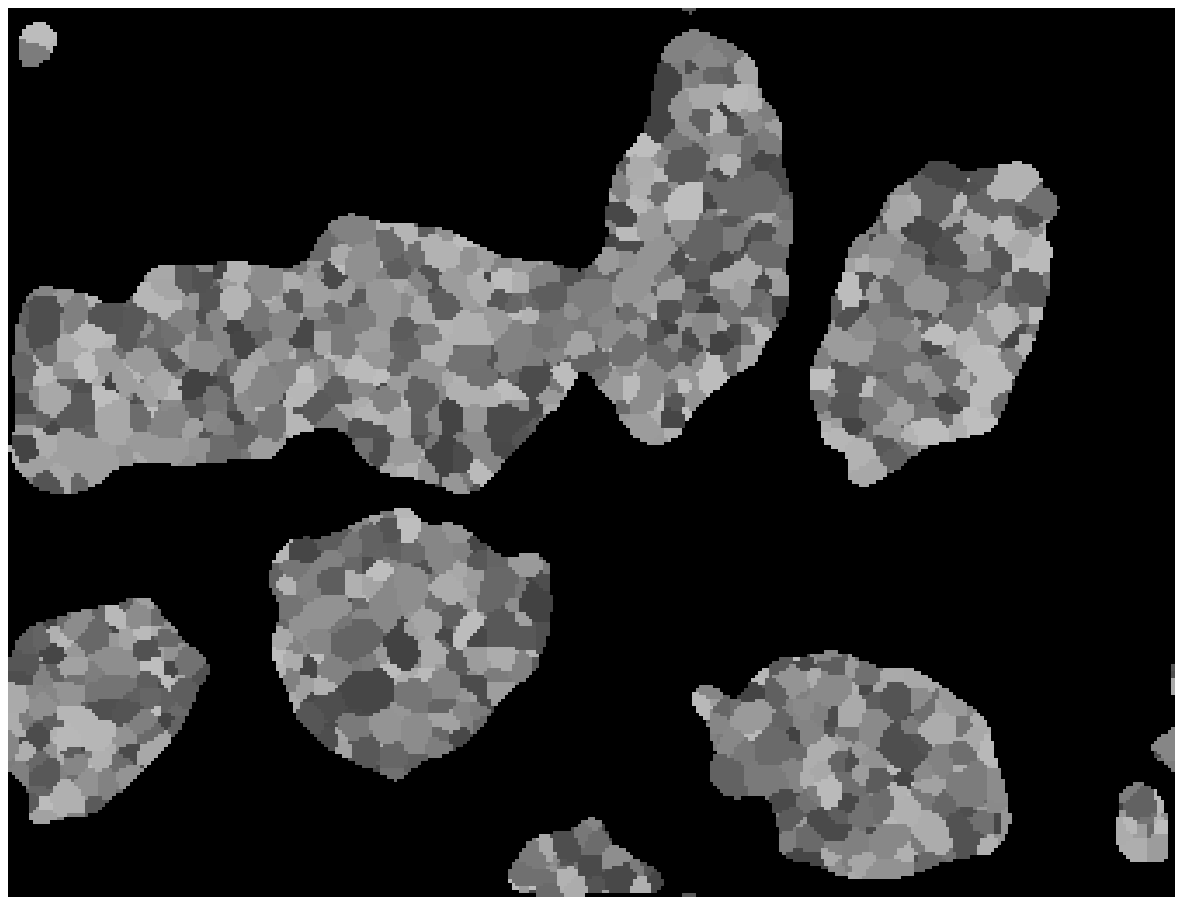}
&\includegraphics[width=0.2\textwidth]{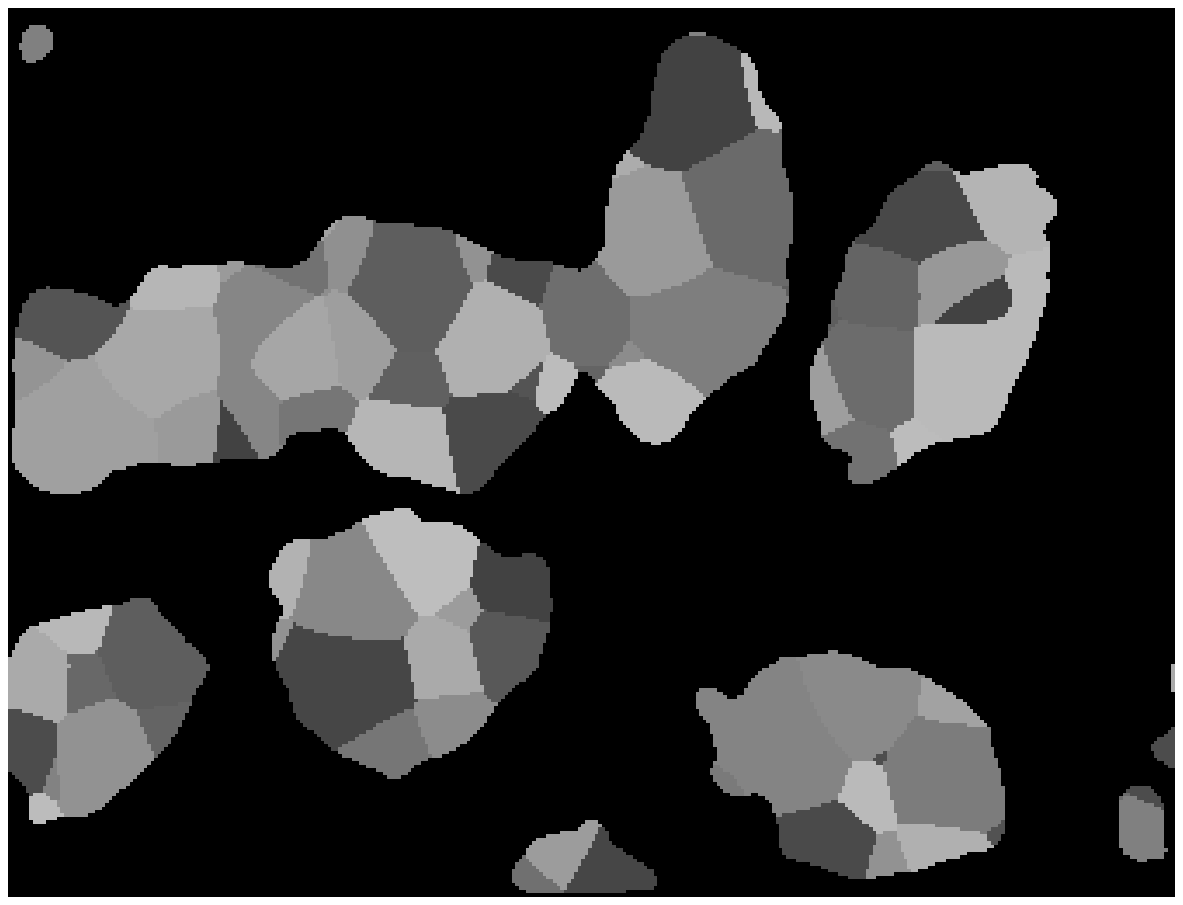}
&\includegraphics[width=0.2\textwidth]{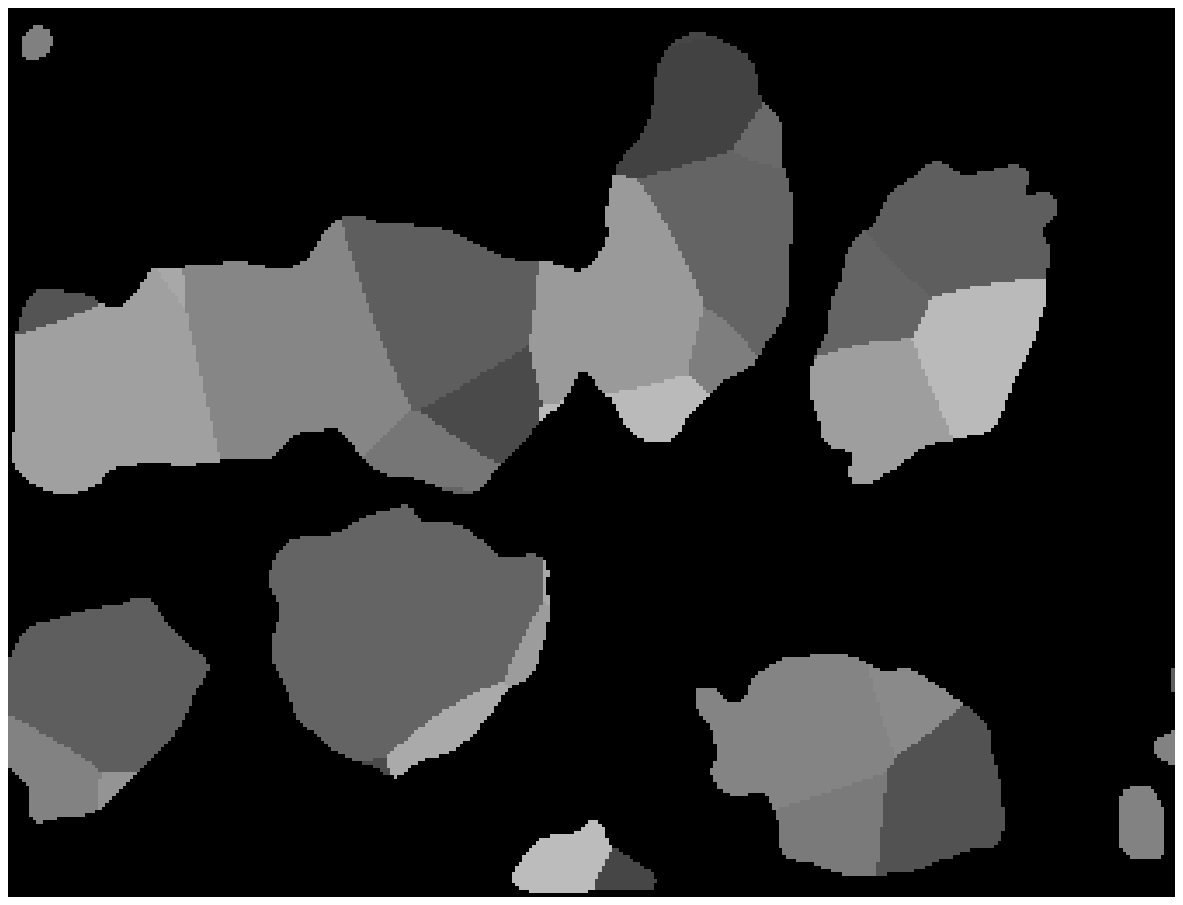}\\
{$i=8$}
& \includegraphics[width=0.2\textwidth]{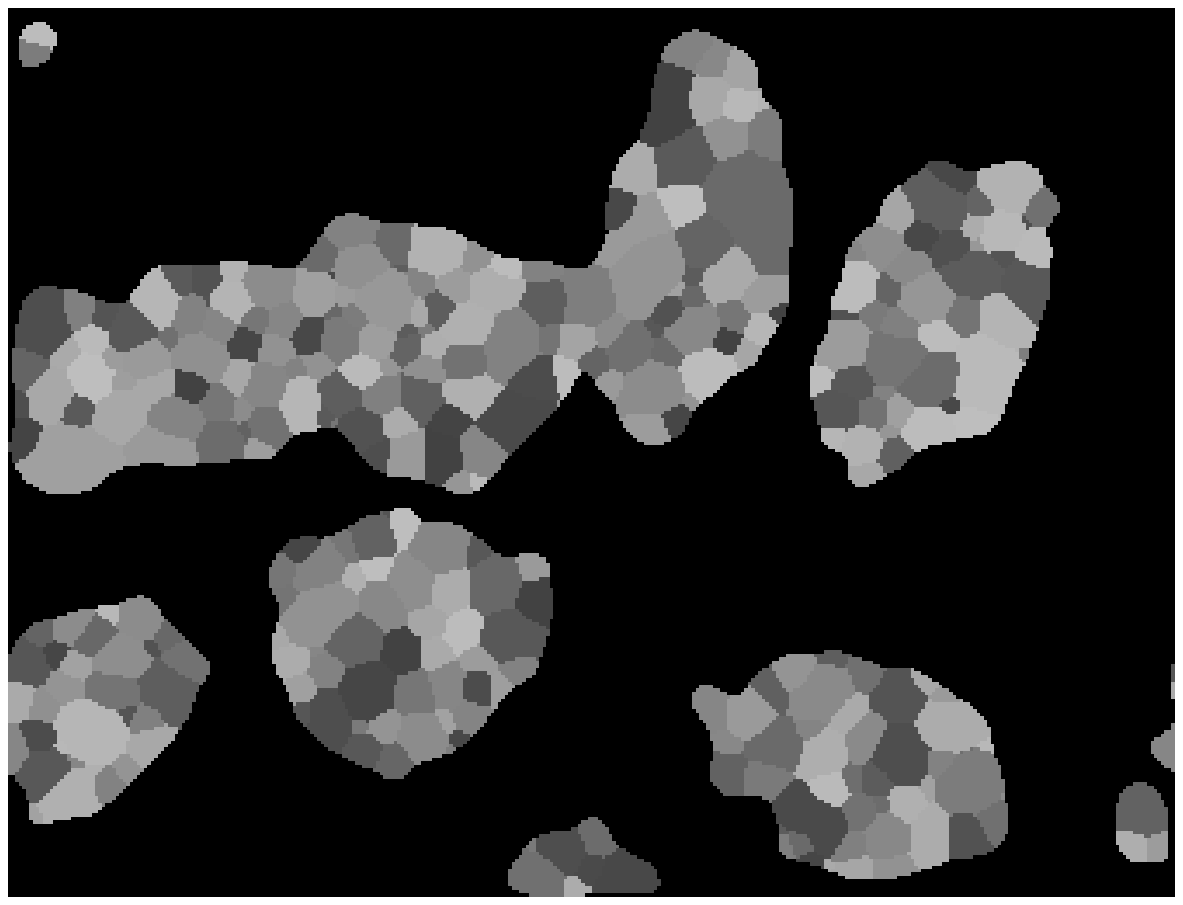}
& \includegraphics[width=0.2\textwidth]{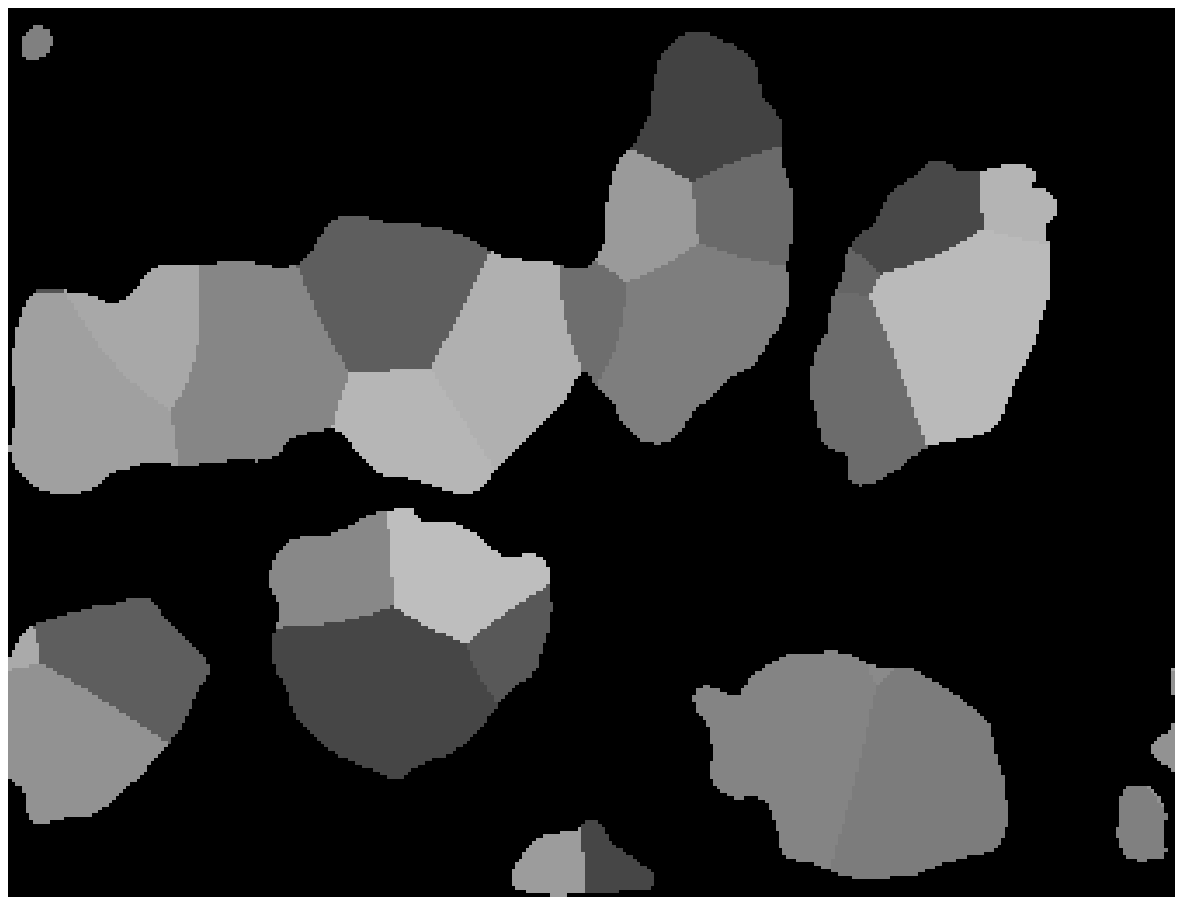}
& \includegraphics[width=0.2\textwidth]{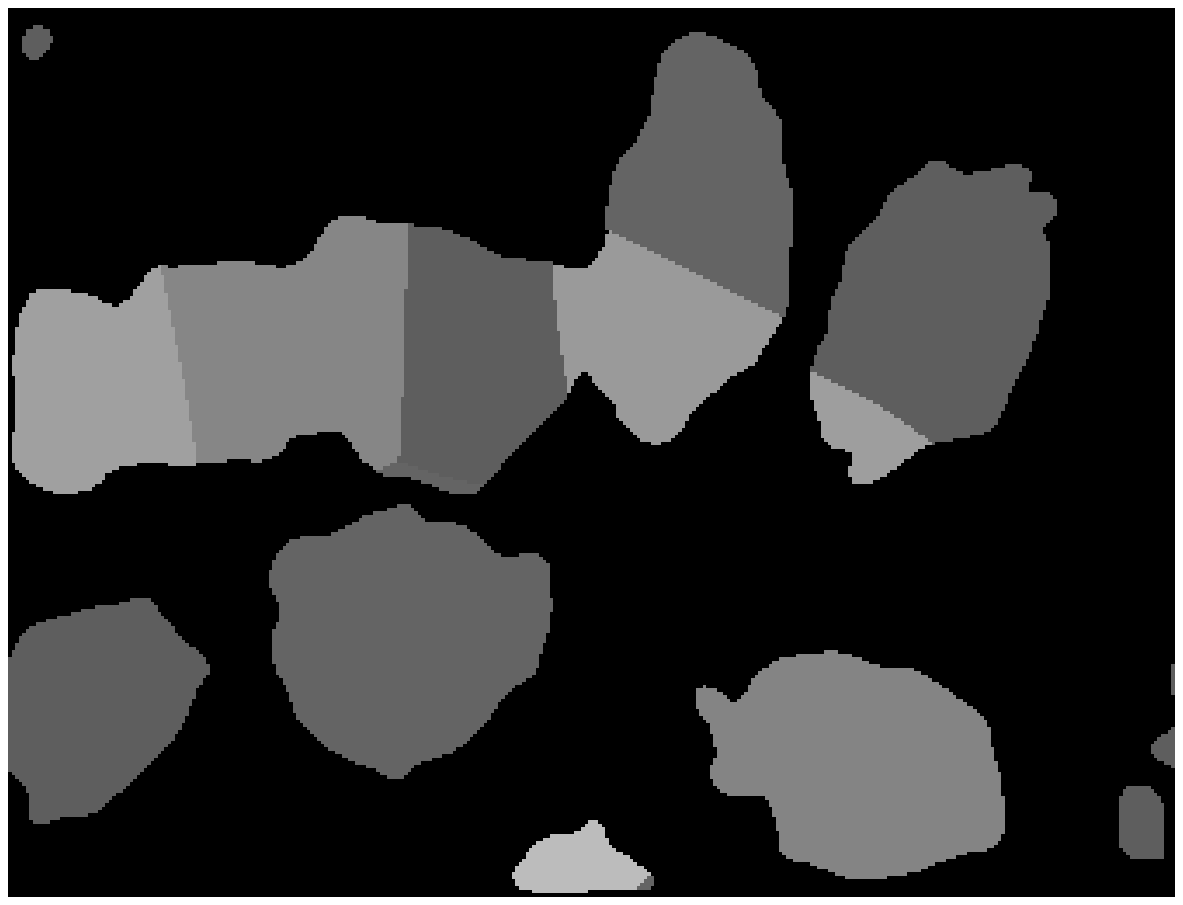}\\
{$i=14$}
& \includegraphics[width=0.2\textwidth]{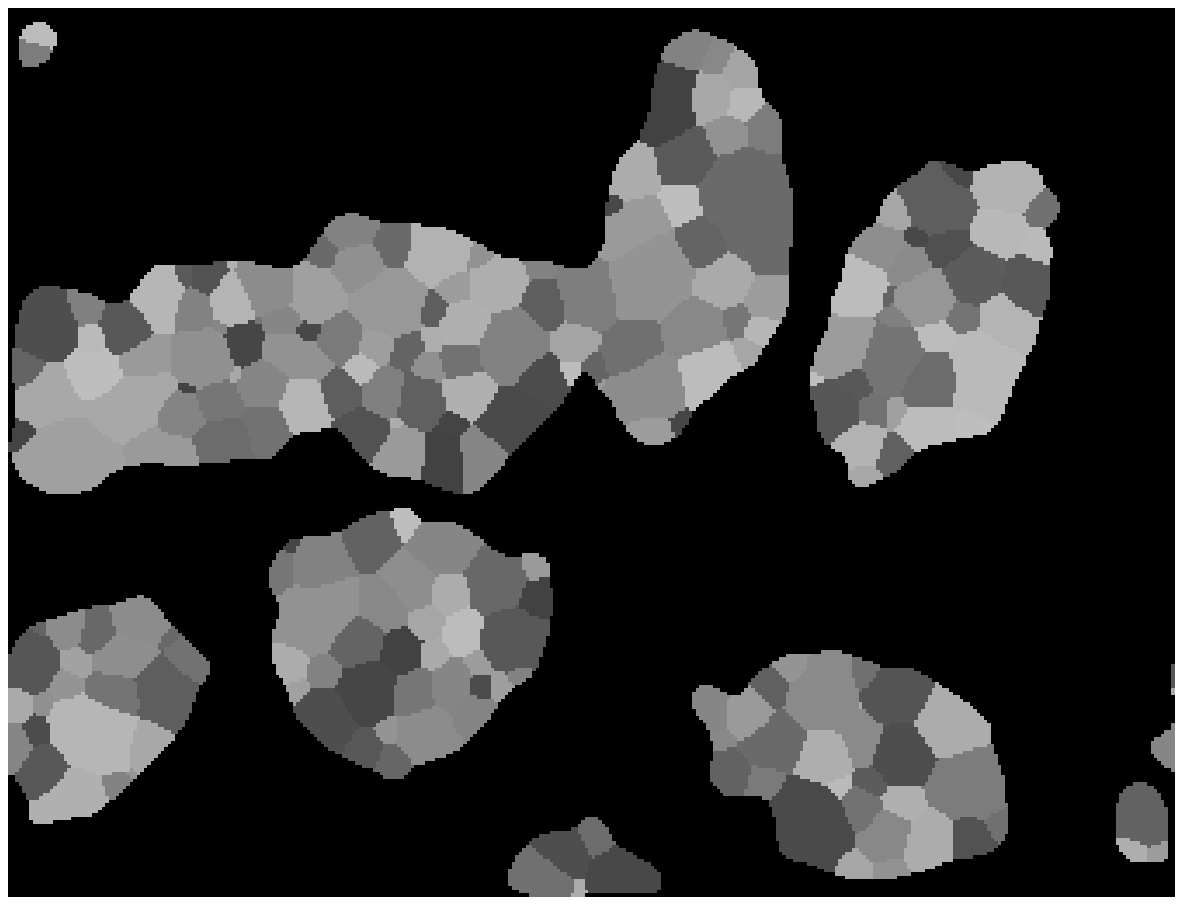}
& \includegraphics[width=0.2\textwidth]{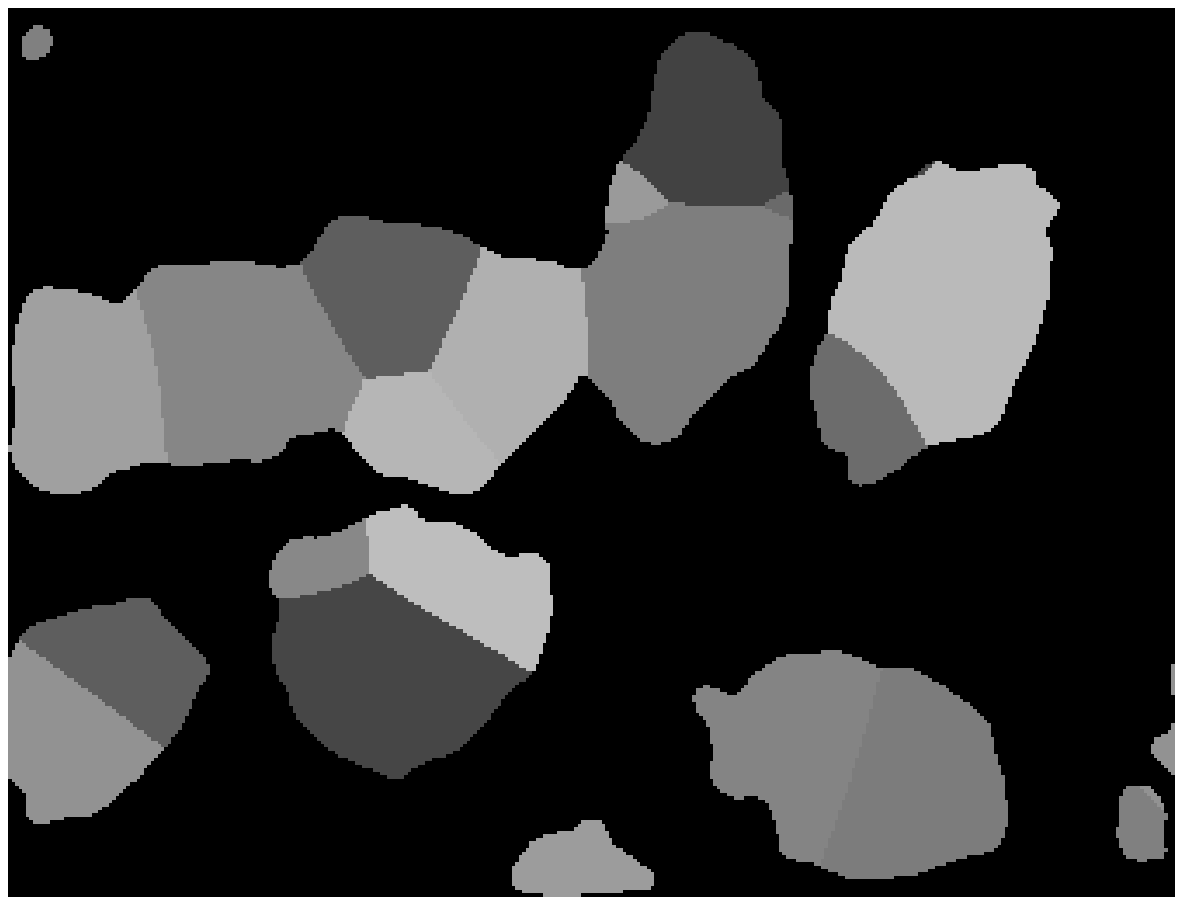}
& \includegraphics[width=0.2\textwidth]{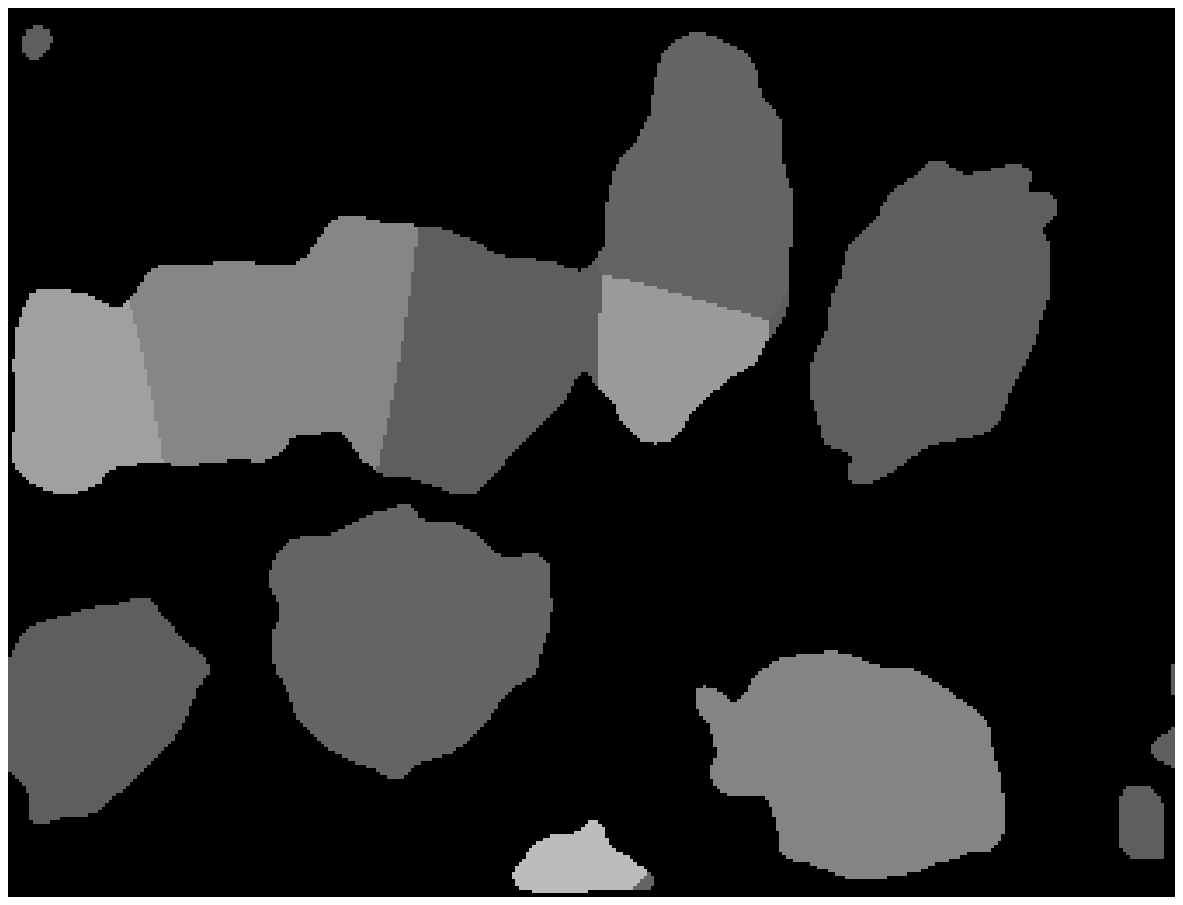}\\
Final state 
& \includegraphics[width=0.2\textwidth]{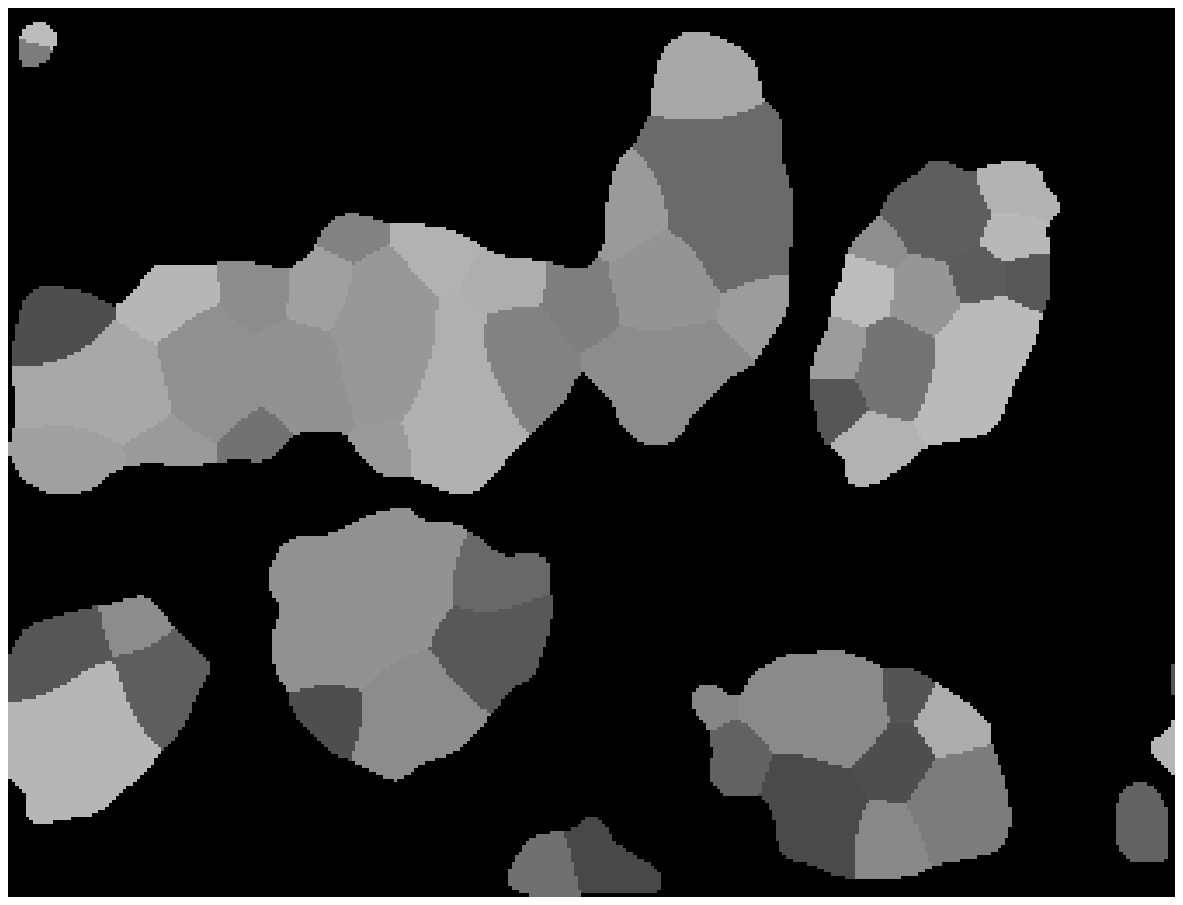}
& \includegraphics[width=0.2\textwidth]{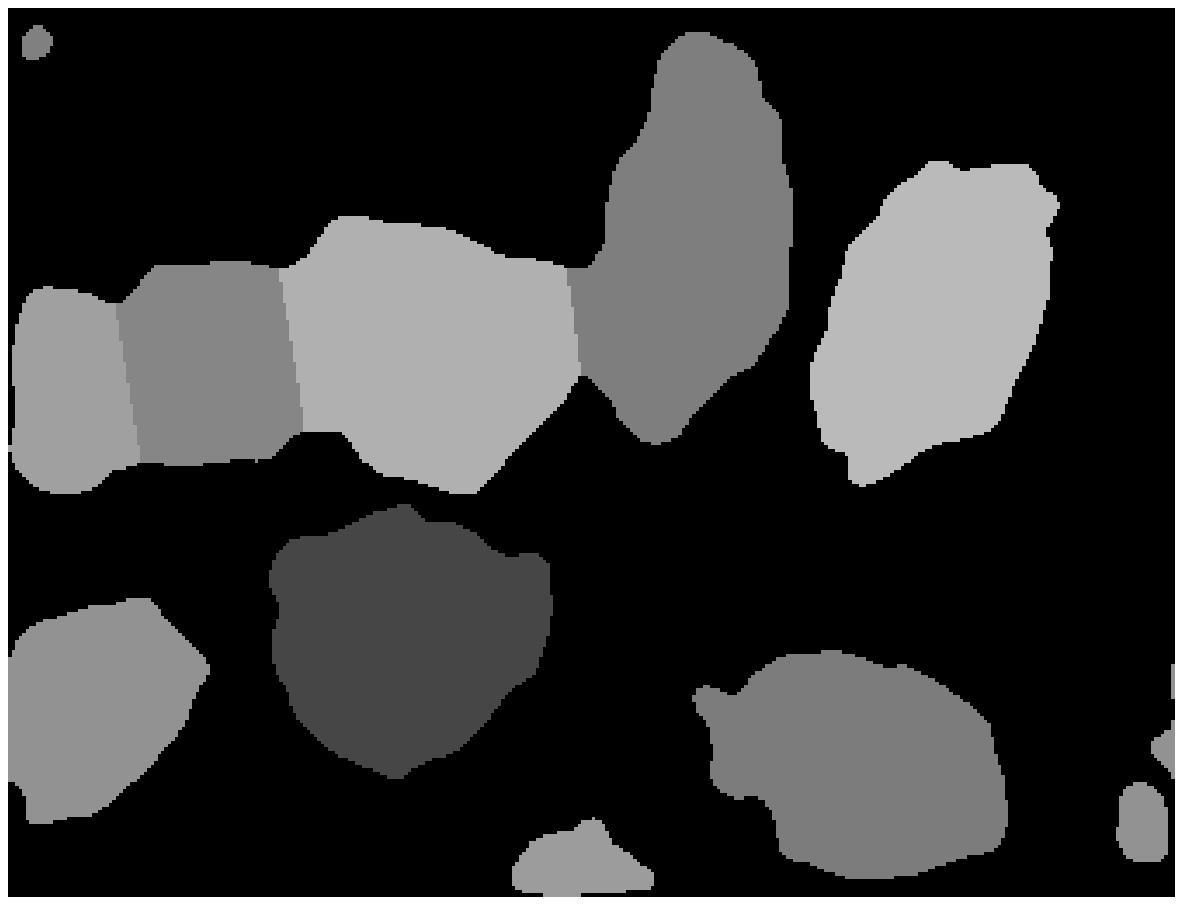}
& \includegraphics[width=0.2\textwidth]{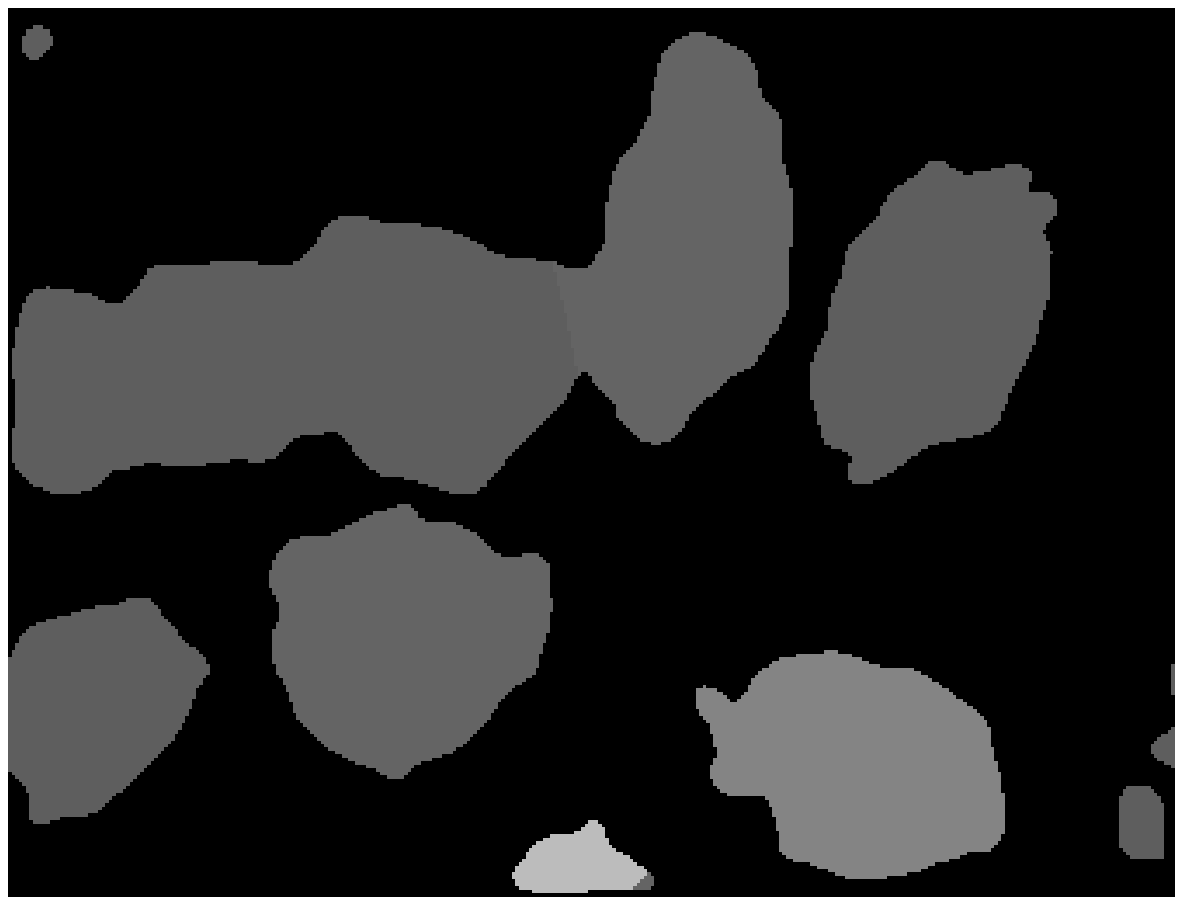}
\end{tabular}
\label{tab:gt}
}
\vspace{0.1in}
\caption{ \label{fig:AM} {\small Active mask segmentation of punctate patterns of proteins \citep{Srinivasa:09}.  (a) Original image (courtesy of Linstedt Lab \citep{LinstedtLab:web}) (b) Initialization using M=256 random masks. After one iteration of~\eqref{eq:am}, the background is separated from the foreground by the region-based skew function $R_1$. (c) Segmentation results using various scales of the voting filter at iterations 2, 8, 14 and at convergence. The first row shows a cross section of three Gaussian filters (scale=4, 16 and 32 respectively). The second row shows the segmentation result after two iterations of the algorithm. When scale=4, we observe a large number of small regions in the foreground. This is contrasted by the fewer number of regions when scale=16 and scale=32.  Subsequent rows represent the state of the system at various stages of evolution. The last row represents the convergence states. Note that the algorithm converges regardless of the scale chosen, but the segmentation is only biologically meaningful at the proper scale of 16; at scales=4, 32 the cells are oversegmented or undersegmented, respectively.  
}}
\end{figure*}

The purpose of this paper is to provide a rigorous investigation of the convergence behavior of the AM algorithm.  To be precise, we note that in a real-world implementation the AM algorithm occasionally fails to converge to a $\psi$ which is biologically meaningful.  However, even in these cases, the algorithm always seems to converge to something.  Indeed, when $\psi_{0}$ and the $R_m$'s are chosen at random, experimentation reveals that the repeated application of~\eqref{eq:am} always seems to eventually produce $\psi_{i}$ such that $\psi_{i+1}=\psi_{i}$.  At the same time, a simple example tempers one's expectations: taking $\Omega=\Int_4$, $M=2$, $w=\delta_{-1}+\delta_0+\delta_1$ and $R_1=R_2=0$, we see that AM will not always converge, as repeatedly applying \eqref{eq:am} to $\psi_{0}=\delta_0+\delta_2$ produces the endless $2$-cycle $\delta_0+\delta_2\mapsto\delta_1+\delta_3\mapsto\delta_0+\delta_2$.  In summary, even though random experimentation indicates that AM will almost certainly converge, there exist trivial examples which show that it will not always do so.  The central question that this paper seeks to address is therefore:
\begin{quote}
\begin{center}
Under what conditions on $g$ and $\{R_m\}_{m=1}^M$ will the AM algorithm always converge to a fixed point of \eqref{eq:am}~?
\end{center}
\end{quote}
We  show that when $g$ is an even function, AM will either converge to a fixed point or will get stuck in a $2$-cycle; no higher-order cycles are possible.  We can further rule out $2$-cycles whenever $g$ is taken so that the convolutional operator $f\mapsto f*g$ is positive semidefinite.  The following is a compilation of these results:
\begin{theorem}
\label{thm:suff}
Given any $\Omega:=\prod_{d=1}^D\Int_{N_d}$, initial segmentation $\psi_{0}:\Omega\rightarrow\{1,\dotsc,M\}$ and any real-valued functions $\{R_m\}_{m=1}^M$ over $\Omega$, the Active Mask algorithm, namely the repeated application of~\eqref{eq:am}, will always converge to a fixed point of~\eqref{eq:am} provided the discrete Fourier transform of $g$ is nonnegative and even.
\end{theorem}
A preliminary version of the results in this paper appears in the conference proceeding~\cite{BalcanSFK:10}.  Though the specific AM algorithm was introduced in~\cite{Srinivasa:09}, iterative lowpass filtering has long been a subject of interest in applied harmonic analysis, having deep connections to continuous-domain ideas such as diffusion and the maximum principle~\citep{Nirenbarg:53}.  For instance, \cite{PeronaM:90} gives an edge detection application of a discretized version of these ideas. Meanwhile,~\citep{Koenderink:84} gives diffusion-inspired conditions under which lowpass filtering is guaranteed to produce a coarse version of a given image.  One way to prove the convergence of iterative convolution-thresholding schemes is to show that lowpass filtering decreases the number of zero-crossings in a signal; such a condition is equivalent to a version of the maximum principle~\citep{Hummel:87}.  More recently, the continuous-domain version of~\eqref{eq:ct} has been used to model the motion of interfaces between media~\citep{Ruuth:98,RuuthM:01}; in that setting,~\eqref{eq:ct} is known to converge if $M=2$.   Since the AM algorithm is iterative, many of the proof techniques we use here were adapted from \textit{majority cellular automata} (MCA), a well-studied class of discrete dynamical systems.  Indeed, theoretical guarantees on the convergence of a symmetric class of MCA have been known for several decades; see \citep{Goles:90,PS:83}, and references therein.  Such results were recently generalized to a quasi-symmetric class via the use of Lyapunov functionals~\citep{YN:09}.  Whereas much of traditional MCA theory focuses on the convergence of repeated applications of~\eqref{eq:ct}, our work differs due to the presence of the additive $R_m$ terms in~\eqref{eq:am}.

The paper is organized as follows.  In the next section, we use an MCA formulation of AM to prove our main convergence results.  In Section 3, we then briefly discuss the generalization of our main results to a less elegant yet more realistic version of~\eqref{eq:am} involving noncircular convolution.  We conclude in Section 4 with some examples illustrating our main results, as well as some experimental results indicating the AM algorithm's rate of convergence.

\section{Active Masks as a Majority Cellular Automaton}
\label{sec:cellauto}

Cellular automata are self-evolving discrete dynamical systems \citep{Goles:90}.  They have been applied in various fields such as statistical physics, computational biology, and the social sciences.  A tremendous amount of work in this area has focused on studying the convergence behavior of various types of automata. In this section, we formulate the AM algorithm~\eqref{eq:am} as an MCA in order to facilitate our understanding of its convergence behavior.  To be precise, we consider a generalization of~\eqref{eq:am} in which the convolutional operator $f\mapsto f*g$ may more broadly be taken to be any linear operator $A$ from $\ell(\Omega):=\set{f|f:\Omega\rightarrow\Reals}$ into itself:
\begin{equation}
\label{eq:am generalized}
\psi_{i}(n)
:=\min\Bigparen{\argmax\limits_{m}\bigbracket{(A\mu_{m}^{(i-1)})(n)+R_{m}(n)}},
\qquad \mu_m^{(i-1)}:=\left\{\begin{array}{ll}1,&\psi_{i-1}(n)=m,\\0,&\psi_{i-1}(n)\neq m.\end{array}\right.
\end{equation}
Here, the contribution of mask $m$ in deciding the outcome at location $n$ at iteration $i$ is $(A\mu_{m}^{(i-1)})(n)$, and any ties are broken by choosing the smallest $m$ corresponding to a maximal element.  Note that given any initial segmentation $\psi_0$, applying~\eqref{eq:am generalized} \textit{ad infinitum} produces a sequence $\{\psi_i\}_{i=0}^{\infty}$.  However, as there are only $M^N$ distinct possible configurations for $\psi:\Omega\rightarrow\{1,\dotsc,M\}$, this sequence must eventually repeat itself.  Indeed, taking minimal indices $i_0$ and $K>0$ such that $\psi_{i_0+K}=\psi_{i_0}$, the deterministic nature of~\eqref{eq:am generalized} implies that $\psi_{i+K}=\psi_i$ for all $i\geq i_0$.  The finite sequence $\{\psi_i\}_{i=i_0}^{i_0+K-1}$ is called a \textit{cycle} of~\eqref{eq:am generalized} of \textit{length} $K$.  Note that $\{\psi_i\}_{i=0}^{\infty}$ converges if and only if $K=1$, which happens precisely when $\psi_{i_0}$ is a fixed point of~\eqref{eq:am generalized}.

Thus, from this perspective, proving that \eqref{eq:am generalized} always converges is equivalent to proving that $K=1$ regardless of one's choice of $\psi_0$.  The following result goes a long way towards this goal, showing that if $A$ is self-adjoint, then for any $\psi_0$ we have that the resulting $K$ is necessarily $1$ or $2$.  That is, if $A$ is self-adjoint, then for any $\psi_0$, the sequence $\set{\psi_i}_{i=0}^{\infty}$ will either converge in a finite number of iterations, or it will eventually come to a point where it forever oscillates between two distinct configurations $\psi_{i_0}$ and $\psi_{i_0+1}$.
\begin{theorem}
\label{thm:sym}
If $A$ is self-adjoint, then for any $\psi_0$, the cycle length $K$ of \eqref{eq:am generalized} is either 1 or 2.
\end{theorem}
\begin{proof}
As we are not presently concerned with the rate of convergence of~\eqref{eq:am generalized}, but rather the question of whether it does converge, we may assume without loss of generality that $\{\psi_i\}_{i=0}^{\infty}$ has already entered its cycle.  That is, we reindex so that $\psi_0$ is the beginning of the $K$-cycle, and heretofore regard all iteration indices as members of the cyclic group $\Int_{K}$.  We argue by contrapositive, assuming $K>2$ and concluding that $A$ is not self-adjoint.  For any $i=1,\dotsc,K$, \eqref{eq:am generalized} is equivalent to the system of inequalities:
\begin{subequations}
\begin{align}
\label{subeq:1}& (A\mu_{\psi_{i}(n)}^{(i-1)})(n)+R_{\psi_{i}(n)}(n)>(A\mu_{m}^{(i-1)})(n)+R_{m}(n) &  & \textrm{if}~~1\leq m<\psi_{i}(n),\\
\label{subeq:2}& (A\mu_{\psi_{i}(n)}^{(i-1)})(n)+R_{\psi_{i}(n)}(n)\geq(A\mu_{m}^{(i-1)})(n)+R_{m}(n) &  & \textrm{if}~~\psi_{i}(n)\leq m\leq M.
\end{align}
\end{subequations}
Here, \eqref{subeq:2} follows from the fact that $\psi_i(n)$ is a value of $m$ that maximizes $(A\mu_{m}^{(i-1)})(n)+R_{m}(n)$.  Moreover, in the event of a tie, $\psi_i(n)$ is chosen to be the least of all such maximizing $m$, yielding the strict inequality in \eqref{subeq:1}.  For any $i$ and $n$, picking $m=\psi_{i-2}(n)$ in~\eqref{subeq:1} and~\eqref{subeq:2} leads to the subsystem of inequalities:
\begin{subequations}
\label{eq:red-syst}
\begin{align}
& (A\mu_{\psi_{i}(n)}^{(i-1)})(n)-(A\mu_{\psi_{i-2}(n)}^{(i-1)})(n)+R_{\psi_{i}(n)}(n)-R_{\psi_{i-2}(n)}(n)>0    &  & \textrm{if}~\psi_{i-2}(n)<\psi_{i}(n),\\
& (A\mu_{\psi_{i}(n)}^{(i-1)})(n)-(A\mu_{\psi_{i-2}(n)}^{(i-1)})(n)+R_{\psi_{i}(n)}(n)-R_{\psi_{i-2}(n)}(n)\geq0 &  & \textrm{if}~~\psi_{i-2}(n)\geq\psi_{i}(n).\end{align}
\end{subequations}
Now since $K>2$, there exists a pixel $n$ for which $\{\psi_0(n),\psi_1(n),\dotsc,\psi_{K-1}(n)\}$ is not of the form $\{a,a,\ldots,a\}$ nor of the form $\{a,b,a,b,\ldots a,b\}$.  At such an $n$, there must exist an $i$ such that $\psi_{i-2}(n)<\psi_{i}(n)$.  Consequently, at least one inequality in \eqref{eq:red-syst} is strict.  Thus, summing~\eqref{eq:red-syst} over all pixels $n$ and all cycle indices $i$ yields:
\begin{equation*}
0<\sum_{i\in\Int_{K}}\sum_{n\in\Omega}(A\mu_{\psi_{i}(n)}^{(i-1)})(n)-\sum_{i\in\Int_{K}}\sum_{n\in\Omega}(A\mu_{\psi_{i-2}(n)}^{(i-1)})(n)+\sum_{i\in\Int_{K}}\sum_{n\in\Omega}R_{\psi_{i}(n)}(n)-\sum_{i\in\Int_{K}}\sum_{n\in\Omega}R_{\psi_{i-2}(n)}(n).
\end{equation*}
Since $\Int_K$ is shift-invariant, $\displaystyle\sum_{i\in\Int_{K}}R_{\psi_{i}(n)}(n)=\sum_{i\in\Int_{K}}R_{\psi_{i-2}(n)}(n)$ for any $n\in\Omega$, reducing the previous equation to:
\begin{equation}
\label{eq:sum}
0
<\sum_{i\in\Int_{K}}\sum_{n\in\Omega}(A\mu_{\psi_{i}(n)}^{(i-1)})(n)-\sum_{i\in\Int_{K}}\sum_{n\in\Omega}(A\mu_{\psi_{i-2}(n)}^{(i-1)})(n)
=\sum_{i\in\Int_{K}}\sum_{n\in\Omega}(A\mu_{\psi_{i}(n)}^{(i-1)})(n)-\sum_{i\in\Int_{K}}\sum_{n\in\Omega}(A\mu_{\psi_{i-1}(n)}^{(i)})(n),
\end{equation}
where the final equality also follows from the shift-invariance of $\Int_K$.  To continue, note that for any $i,j\in\Int_K$ we have $\mu_m^{(j)}=1$ if and only if $\psi_j(n)=m$ and so:
\begin{equation}
\label{eq:trace}
\sum_{n\in\Omega}(A\mu_{\psi_j(n)}^{(i)})(n)
=\sum_{n\in\Omega}\sum_{m=1}^M(A\mu_m^{(i)})(n)\mu_m^{(j)}(n)
=\sum_{m=1}^M\ip{A\mu_m^{(i)}}{\mu_m^{(j)}},
\end{equation}
where \smash{$\displaystyle\ip{f}{g}:=\sum_{n\in\Omega}f(n)g(n)$} is the standard real inner product over $\Omega$.  Using~\eqref{eq:trace} in~\eqref{eq:sum} gives:
\begin{equation*}
0<\sum_{i\in\Int_{K}}\sum_{m=1}^M\ip{A\mu_m^{(i-1)}}{\mu_m^{(i)}}-\sum_{i\in\Int_{K}}\sum_{m=1}^M\ip{A\mu_m^{(i)}}{\mu_m^{(i-1)}}
=\sum_{i\in\Int_{K}}\sum_{m=1}^M\ip{(A-A^*)\mu_m^{(i-1)}}{\mu_m^{(i)}},
\end{equation*}
implying $A-A^*\neq0$, and so $A$ is not self-adjoint.
\end{proof}
Theorem~\ref{thm:sym} has strong implications for the AM algorithm~\eqref{eq:am}.  Indeed, it is well-known that if $g$ is real-valued, then the adjoint of the convolutional operator $A f=f*g$ is $A^* f=f*\tilde{g}$ where $\tilde{g}(n)=g(-n)$ is the \textit{reversal} of $g$.  As such, if $g$ is an even function, Theorem~\ref{thm:sym} guarantees that AM will always either converge or enter a $2$-cycle.  We now build on the techniques of the previous proof to find additional restrictions on $A$ which suffice to guarantee convergence:
\begin{theorem}
\label{thm:fp}
If $A$ is self-adjoint and $\ip{A f}{f}\geq0$ for all $f:\Omega\rightarrow\{0,\pm1\}$, then~\eqref{eq:am generalized} always converges.
\end{theorem}

\begin{proof}
In light of Theorem~\ref{thm:sym}, our goal is to rule out cycles of length $K=2$.  We argue by contrapositive.  That is, we assume that there exist two distinct configurations $\psi_{0}$ and $\psi_{1}$ which are successors of each other, and will use this fact to produce $f:\Omega\rightarrow\{0,\pm1\}$ such that $\ip{A f}{f}<0$.  Substituting $i=0$ and $m=\psi_{1}(n)$ into~\eqref{subeq:1} and~\eqref{subeq:2} yields:
\begin{subequations}
\begin{align}
\label{eq:new-red-syst0:a}
& (A\mu_{\psi_{0}(n)}^{(1)})(n)-(A\mu_{\psi_{1}(n)}^{(1)})(n)+R_{\psi_{0}(n)}(n)-R_{\psi_{1}(n)}(n)>0 &  & \textrm{if}~~\psi_{1}(n)<\psi_{0}(n),\\
\label{eq:new-red-syst0:b}
& (A\mu_{\psi_{0}(n)}^{(1)})(n)-(A\mu_{\psi_{1}(n)}^{(1)})(n)+R_{\psi_{0}(n)}(n)-R_{\psi_{1}(n)}(n)\geq0 &  & \textrm{if}~~\psi_{1}(n)\geq\psi_{0}(n).
\end{align}
\end{subequations}
Similarly, letting $i=1$ and $m=\psi_{0}(n)$ into~\eqref{subeq:1} and~\eqref{subeq:2} yields:
\begin{subequations}
\begin{align}
\label{eq:new-red-syst1:a}
&(A\mu_{\psi_{1}(n)}^{(0)})(n)-(A\mu_{\psi_{0}(n)}^{(0)})(n)+R_{\psi_{1}(n)}(n)-R_{\psi_{0}(n)}(n)>0 &  & \textrm{if}~~\psi_{0}(n)<\psi_{1}(n),\\
\label{eq:new-red-syst1:b}
&(A\mu_{\psi_{1}(n)}^{(0)})(n)-(A\mu_{\psi_{0}(n)}^{(0)})(n)+R_{\psi_{1}(n)}(n)-R_{\psi_{0}(n)}(n)\geq0 &  & \textrm{if}~~\psi_{0}(n)\geq\psi_{1}(n).
\end{align}
\end{subequations}
Since $\psi_0$ and $\psi_1$ are distinct, there exists $n_0\in\Omega$ such that $\psi_0(n_0)\neq\psi_1(n_0)$.  If $\psi_0(n_0)<\psi_1(n_0)$,  we sum~\eqref{eq:new-red-syst0:b} and~\eqref{eq:new-red-syst1:a} over all $n\in\Omega$.  If on the other hand $\psi_0(n_0)>\psi_1(n_0)$, we sum~\eqref{eq:new-red-syst0:a} and~\eqref{eq:new-red-syst1:b} over all $n\in\Omega$.  Either way, we obtain:
\begin{equation*}
0<\sum_{n=1}^N\bigbracket{(A\mu_{\psi_{0}(n)}^{(1)})(n)-(A\mu_{\psi_{1}(n)}^{(1)})(n)+(A\mu_{\psi_{1}(n)}^{(0)})(n)-(A\mu_{\psi_{0}(n)}^{(0)})(n)}.
\end{equation*}
Applying~\eqref{eq:trace} four times then gives:
\begin{equation*}
0
<\sum_{m=1}^M\bigbracket{\ip{A\mu_m^{(1)}}{\mu_m^{(0)}}-\ip{A\mu_m^{(1)}}{\mu_m^{(1)}}+\ip{A\mu_m^{(0)}}{\mu_m^{(1)}}-\ip{A\mu_m^{(0)}}{\mu_m^{(0)}}}
=-\sum_{m=1}^{M}\bigip{A(\mu_m^{(1)}-\mu_m^{(0)})}{(\mu_m^{(1)}-\mu_m^{(0)})}.
\end{equation*}
As such, there exists at least one index $m_0$ such that $0>\bigip{A(\mu_{m_0}^{(1)}-\mu_{m_0}^{(0)})}{(\mu_{m_0}^{(1)}-\mu_{m_0}^{(0)})}$; choose $f$ to be $\mu_{m_0}^{(1)}-\mu_{m_0}^{(0)}$.
\end{proof}
The most obvious way to ensure that $\ip{A f}{f}\geq0$ for all $f:\Omega\rightarrow\{0,\pm1\}$ is for $A$ to be positive semidefinite, that is, $\ip{A f}{f}\geq0$ for all $f:\Omega\rightarrow\Reals$.  This in turn can be guaranteed by taking $A$ to be diagonally dominant with nonnegative diagonal entries, via the Gershgorin circle Theorem \cite{Golub96book}. Note that in fact strict diagonal dominance guarantees that iterative voting~\eqref{eq:ct} always converges in one iteration.  More interesting examples can be found in the special case where $A$ is a convolutional operator $Af=f*g$.  Indeed, letting $\rmF$ be the standard non-normalized discrete Fourier transform (DFT) over $\Omega$, we have:
\begin{equation}
\label{eq:posdef}
\ip{Af}{f}=\ip{f*g}{f}=\frac1N\bigip{\rmF(f*g)}{\rmF f}=\frac1N\bigip{(\rmF f)(\rmF g)}{\rmF f}=\frac1N\sum_{n\in\Omega}(\rmF g)(n)\bigabs{(\rmF f)(n)}^2.
\end{equation}
As such, if $g:\Omega\rightarrow\Reals$ is even and $(\rmF g)(n)\geq0$ for all $n\in\Omega$, then $A$ is self-adjoint positive semidefinite.   Moreover, it is well-known that $g$ is real-valued and even if and only if $\rmF g$ is also real-valued and even.  Thus, $A$ is self-adjoint positive semidefinite provided $\rmF g$ is nonnegative and even.  For such $g$, Theorem~\ref{thm:fp} guarantees that the AM algorithm~\eqref{eq:am} will always converge.  These facts are summarized in Theorem~\ref{thm:suff}, which is stated in the introduction.  Examples of $g$ that satisfy these hypotheses are given in Section~\ref{sec:experiments}.

We emphasize that Theorem~\ref{thm:fp} does not require $A$ to be positive semidefinite, but rather only that $\ip{A f}{f}\geq0$ for all $f:\Omega\rightarrow\{0,\pm1\}$.  In the case of convolutional operators, this means we truly only need~\eqref{eq:posdef} to hold for such $f$'s.  As such, it may be overly harsh to require that $(Fg)(n)\geq0$ for all $n\in\Omega$.  Unfortunately, the problem of characterizing such $g$'s appears difficult, as we could find no useful frequency-domain characterizations of $\{0,\pm1\}$-valued functions.  A spatial domain approach is more encouraging: when $\Omega=\Int_N$, writing $f:\Omega\rightarrow\set{0,\pm1}$ as the difference of two characteristic functions $\chi_{I_1},\chi_{I_2}:\Omega\rightarrow\set{0,1}$,
we have:
\begin{equation*}
\ip{Af}{f}
=\bigip{A(\chi_{I_1}-\chi_{I_2})}{(\chi_{I_1}-\chi_{I_2})}
=\ssum(A_{1,1})+\ssum(A_{2,2})-\ssum(A_{1,2})-\ssum(A_{2,1}),
\end{equation*}
where $\ssum(A_{i,j})$ denotes the sum of all entries of the submatrix of $A$ consisting of rows from $I_i$ and columns from $I_j$.  As such, the condition of Theorem~\ref{thm:fp} reduces to showing that $0\leq \ssum(A_{1,1})+\ssum(A_{2,2})-\ssum(A_{1,2})-\ssum(A_{2,1})$ for all choices of subsets $I_i$ and $I_j$ of $\Int_N$.

We conclude this section by noting that~\eqref{eq:am generalized} is similar to \textit{threshold cellular automata} (TCA)~\cite{Goles:90,Goles85dam}.  In fact,~\eqref{eq:am generalized} is equivalent to TCA in the special case of $M=2$; in this case, $\mu_{0}^{(i-1)}(n)=1-\mu_{1}^{(i-1)}(n)$ for all $n\in\Omega$, implying:
\begin{align*}
(A\mu_1^{(i-1)})(n)+R_{1}(n)>(A\mu_0^{(i-1)})(n)+R_{0}(n)\quad
&\Longleftrightarrow\quad\bigbracket{A(\mu_{1}^{(i-1)}-\mu_{0}^{(i-1)})}(n)+(R_{1}-R_{0})(n)>0\\
&\Longleftrightarrow\quad\Bigset{A\bigbracket{\mu_{1}^{(i-1)}-(1-\mu_{1}^{(i-1)})}}(n)+(R_{1}-R_{0})(n)>0\\
&\Longleftrightarrow\quad(A\mu_{1}^{(i-1)})(n)+\tfrac12(R_{1}-R_{0}-A1)(n)>0\\
&\Longleftrightarrow\quad(A\mu_{1}^{(i-1)})(n)+b(n)>0,
\end{align*}
where $b(n):=\tfrac12(R_{1}-R_{0}-A1)(n)$.  That is, when $M=2$, the AM algorithm is equivalent to a threshold-like decision.  But whereas the traditional method for proving the convergence of TCA involves associated quadratic Lyapunov functionals~\cite{Goles85dam}, our method for proving the convergence of AM is more direct, being closer in spirit to that of~\cite{PS:83}.

\section{Beyond symmetry}
\label{sec:BeyondSymmetry}

Up to this point, we have focused on the convergence of~\eqref{eq:am generalized} in the special case where $A$ is self-adjoint.  In this section, we discuss how Theorems~\ref{thm:sym} and~\ref{thm:fp} generalize to the case of \textit{quasi-self-adjoint} operators, which arise in real-world implementation of the AM algorithm.  To clarify, up to this point, we have let the image $f$ and weights $g$ be functions over the finite abelian group $\Omega=\prod_{d=1}^D\Int_{N_d}$ and have taken the convolutions in~\eqref{eq:am} and~\eqref{eq:am generalized} to be circular.  In real-world implementation, the use of such circular convolutions can result in poor segmentation, as values at one edge of the image are used to influence the segmentation at the unrelated opposite edge.

One solution to this problem---implemented in~\cite{Srinivasa:09}---is to redefine the set of pixels as a subset $\Omega:=\prod_{d=1}^D[0,N_d)$ of the $D$-dimensional integer lattice $\Int^D$, and regard our image $f$ as a member of $\ell(\Omega):=\set{f:\Int^D\rightarrow\Reals\ |\ f(n)=0\ \forall n\notin\Omega}$.  Here, the label function $\psi$ and masks $\mu_m$ are regarded as $\{1,\dotsc,M\}$- and $\{0,1\}$-valued members of $\ell(\Omega)$, respectively, and the (noncommutative) convolution of any $f,g\in\ell(\Omega)$ with $g\in\ell^2(\Int^D)$ is defined as $f\star g\in\ell(\Omega)$,
\begin{equation}
\label{eq:newconv}
(f\star g)(n):= \frac{(f* g)(n)}{(\chi_{\Omega}* g)(n)},\quad\forall{n}\in\Omega,
\end{equation}
where $\chi_{\Omega}$ is the characteristic function of $\Omega$, and $*$ denotes standard (noncircular) convolution in $\ell^2(\Int^D)$.  For the theory below, we need to place additional restrictions on $g$, namely that it belongs to the class:
\begin{equation*}
\mathcal{G}(\Omega):=\set{g\in\ell^2(\Int^D) : (\chi_{\Omega}* g)(n)>0\ \ \forall n\in\Omega}.
\end{equation*}
In this setting, for a given $g\in\mathcal{G}(\Omega)$, the AM algorithm~\eqref{eq:am} becomes:
\begin{equation}
\label{eq:am,noncircular}
\text{Noncircular Active Masks:}
\qquad\psi_{i}(n)=\argmax\limits_{1\leq m\leq M}~\bigbracket{(\mu_{m}^{(i-1)}\star g)(n)+R_{m}(n)},
\qquad \mu_m^{(i-1)}:=\left\{\begin{array}{ll}1,&\psi_{i-1}(n)=m,\\0,&\psi_{i-1}(n)\neq m.\end{array}\right.
\end{equation}
Note that the use of the $\star$-convolution in~\eqref{eq:am,noncircular} ensures that any ``missing votes'' are not counted in favor of any label $m$.  Moreover, the denominator of~\eqref{eq:newconv} ensures that when $n$ is close to an edge of $\Omega$, the weights in the $g$-neighborhood of $n$ are rescaled so as to always sum to one.  This rescaling ensures that $\sum_{m=1}^{M}(\mu_{m}^{(i)}\star g)(n)=1$ for all $n\in\Omega$, avoiding any need to modify the skew functions $R_m$ near the boundary.

We then ask the question: for what $g$ will \eqref{eq:am,noncircular} always converge?  The key to answering this question is to realize that the $\star$-filtering operation $Af=f\star g$ can be factored as $A=DB$, where $B$ is the standard filtering operator $Bf=f*g$ and $(Df)(n)=\lambda_n f(n)$, where $\lambda_n=[(\chi_{\Omega}* g)(n)]^{-1}$.  Here, $A$, $B$ and $D$ are all regarded as linear operators from $\ell(\Omega)$ into itself.  More generally, we inquire into the convergence of:
\begin{equation}
\label{eq:am generalized,noncircular}
\psi_{i}(n)=\argmax\limits_{1\leq m\leq M}~\bigbracket{(A\mu_{m}^{(i-1)})(n)+R_{m}(n)},
\qquad \mu_m^{(i-1)}:=\left\{\begin{array}{ll}1,&\psi_{i-1}(n)=m,\\0,&\psi_{i-1}(n)\neq m,\end{array}\right.
\end{equation}
where $A=DB$ and $D$ is \textit{positive-multiplicative}, that is, $(Df)(n)=\lambda_n f(n)$ where $\lambda_n>0$ for all $n\in\Omega$.  In particular, we follow~\citep{YN:09} in saying that $A$ is \textit{quasi-self-adjoint} if there exists a positive-multiplicative operator $D$ and a self-adjoint operator $B$ such that $A=DB$.  This definition in hand, we have the following generalization of Theorems~\ref{thm:sym} and~\ref{thm:fp}:
\begin{theorem}
\label{thm:quasi}
Let $A$ be quasi-self-adjoint: $A=DB$ where $D$ is positive-multiplicative and $B$ is self-adjoint. Then for any $\psi_0$, the cycle length $K$ of \eqref{eq:am generalized,noncircular} is either 1 or 2.  Moreover, if $B$ is positive-semidefinite, then \eqref{eq:am generalized,noncircular} always converges.
\end{theorem}
\begin{proof}
We only outline the proof, as it closely follows those of Theorems~\ref{thm:sym} and~\ref{thm:fp}.  Let $(Df)(n)=\lambda_n f(n)$ with $\lambda_n>0$ for all $n\in\Omega$.  We prove the first conclusion by contrapositive, assuming $K>2$.   Rather than summing \eqref{eq:red-syst} over all $n$ and $i$ directly, we instead first divide each instance of \eqref{eq:red-syst} by the corresponding $\lambda_n$, and then sum.  The resulting quantity is analogous to \eqref{eq:sum}:
\begin{equation}
\label{eq:sum, quasi}
0
<\sum_{i\in\Int_{K}}\sum_{n\in\Omega}\frac1{\lambda_n}(A\mu_{\psi_{i}(n)}^{(i-1)})(n)-\sum_{i\in\Int_{K}}\sum_{n\in\Omega}\frac1{\lambda_n}(A\mu_{\psi_{i-1}(n)}^{(i)})(n)
=\sum_{i\in\Int_{K}}\sum_{n\in\Omega}(B\mu_{\psi_{i}(n)}^{(i-1)})(n)-\sum_{i\in\Int_{K}}\sum_{n\in\Omega}(B\mu_{\psi_{i-1}(n)}^{(i)})(n).
\end{equation}
Simplifying the right-hand side of~\eqref{eq:sum, quasi} with~\eqref{eq:trace} quickly reveals that $B$ cannot be self-adjoint, completing this part of the proof.  For the second conclusion, we again prove by contrapositive, assuming $K=2$.  Dividing \eqref{eq:new-red-syst0:a}, \eqref{eq:new-red-syst0:b}, \eqref{eq:new-red-syst1:a} and \eqref{eq:new-red-syst1:b} by $\lambda_n$ and then summing either \eqref{eq:new-red-syst0:a} and \eqref{eq:new-red-syst1:b} over all $n$ or \eqref{eq:new-red-syst0:b} and \eqref{eq:new-red-syst1:a} over all $n$ gives:
\begin{equation*}
0<\sum_{n=1}^N\frac1{\lambda_n}\bigbracket{(A\mu_{\psi_{0}(n)}^{(1)})(n)-(A\mu_{\psi_{1}(n)}^{(1)})(n)+(A\mu_{\psi_{1}(n)}^{(0)})(n)-(A\mu_{\psi_{0}(n)}^{(0)})(n)}
=-\sum_{m=1}^{M}\bigip{B(\mu_m^{(1)}-\mu_m^{(0)})}{(\mu_m^{(1)}-\mu_m^{(0)})},
\end{equation*}
implying $B$ is not positive semidefinite.
\end{proof}
For a result about the convergence of~\eqref{eq:am,noncircular}, we apply Theorem~\ref{thm:quasi} to $A=DB$ where $\lambda_n=[(\chi_{\Omega}* g)(n)]^{-1}$ and $Bf=f*g$.  Note that we must have $g\in\mathcal{G}(\Omega)$ in order to guarantee that $D$ is positive.  Moreover, $B$ is self-adjoint if $g\in\ell^2(\Int^D)$ is even; since $g$ is real-valued, this is equivalent to having its classical Fourier series $\hat{g}\in L^2(\mathbb{T}^D)$ be real-valued and even.  Meanwhile, since:
\begin{equation*}
\ip{Bf}{f}
=\ip{f*g}{f}
=\ip{\hat{f}\hat{g}}{\hat{f}}
=\int_{\mathbb{T}^d}\hat{g}(x)\bigabs{\hat{f}(x)}^2\,\mathrm{d}x,
\end{equation*}
then $B$ is positive semidefinite if $\hat{g}(x)\geq0$ for almost every $x\in\mathbb{T}^D$.  To summarize, we have:
\begin{corollary}
\label{cor:noncirculant}
If the Fourier series of $g\in\mathcal{G}(\Omega)$ is nonnegative and even, then \eqref{eq:am,noncircular} will always converge.
\end{corollary}
In the next section, we discuss how to construct such windows $g$, along with other implementation-related issues.

\section{Examples of Active Masks in practice}
\label{sec:experiments}

In this section we present a few representative and interesting examples of filter-based cellular automata, and discuss their behavior in relation with the results we proved in the previous sections.  We also present some preliminary experimental findings on the rate of convergence of AM.  For ease of understanding, let us for the moment restrict ourselves to circulant iterative voting~\eqref{eq:ct}, namely the version of AM~\eqref{eq:am} in which all the skew functions $R_m$ are identically zero.  The simplest nonzero filter is $g=\delta_0$.  The DFT of $\delta_0$ has constant value $1$, and is therefore nonnegative and even.  As such, Theorem~\ref{thm:suff} guarantees that~\eqref{eq:ct} will always converge.  Of course, we already knew that: since $f*\delta_0=f$ for all $f\in\ell(\Omega)$, \eqref{eq:ct} will always converge in one step; as noted above, the same holds true for any $g$ whose convolutional operator is strictly diagonally dominant with a nonnegative diagonal: $g(0)\geq\sum_{n\neq 0}\abs{g(n)}$.

More interesting examples arise from \textit{box filters}: symmetric cubes of Dirac $\delta$'s.  For instance, fix $N\geq3$ and consider~\eqref{eq:ct} over $\Omega=\Int_N$ where $g=\delta_{-1}+\delta_0+\delta_1$.  Since $g$ is symmetric, Theorem~\ref{thm:sym} guarantees that~\eqref{eq:ct} will either always converge or will enter a $2$-cycle.  However, if $N$ is even, then~\eqref{eq:ct} will not always converge, since $\psi_0=\delta_0+\delta_2+\dots+\delta_{N-2}$ 
generates a $2$-cycle.  This phenomenon is depicted in Figure~\ref{fig:Oscillating states}(a).  This simple example shows that symmetry alone does not suffice to guarantee convergence; one truly needs additional hypotheses on $g$, such as the requirement in Theorem~\ref{thm:suff} that its DFT is nonnegative.  This hypothesis does not hold for $g=\delta_{-1}+\delta_0+\delta_1$, since $(\rmF g)(n)=1+2\cos(\frac{2\pi n}N)$.  Similar issues arise in the two-dimensional setting $\Omega=\Int_{N_1}\times\Int_{N_2}$: both the $3\times 3$ box filter (Moore's automaton, see Figure~\ref{fig:Oscillating states}(b)) and the ``plus'' filter (von Neumann's automaton, see Figure~\ref{fig:Oscillating states}(c)) are symmetric, meaning their cycle lengths are either $1$ or $2$, but neither are positive semidefinite, having DFTs of $[1+2\cos(\frac{2\pi n_1}{N_1})][1+2\cos(\frac{2\pi n_2}{N_2})]$ and $1+2\cos(\frac{2\pi n_1}{N_1})+2\cos(\frac{2\pi n_2}{N_2})$, respectively.  Indeed, when $N_1$ and $N_2$ are even, alternating stripes generate a $2$-cycle for the box filter, while the checkerboard generates a $2$-cycle for the plus filter.
\begin{figure*}
\centering
\subfloat[3-tap filter automaton]{\includegraphics[width=0.65\textwidth]{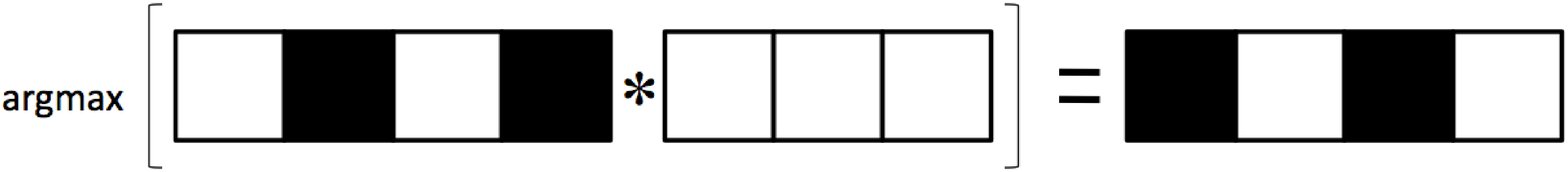}}
\linebreak{}
\subfloat[Moore's automaton]{\includegraphics[width=0.65\textwidth]{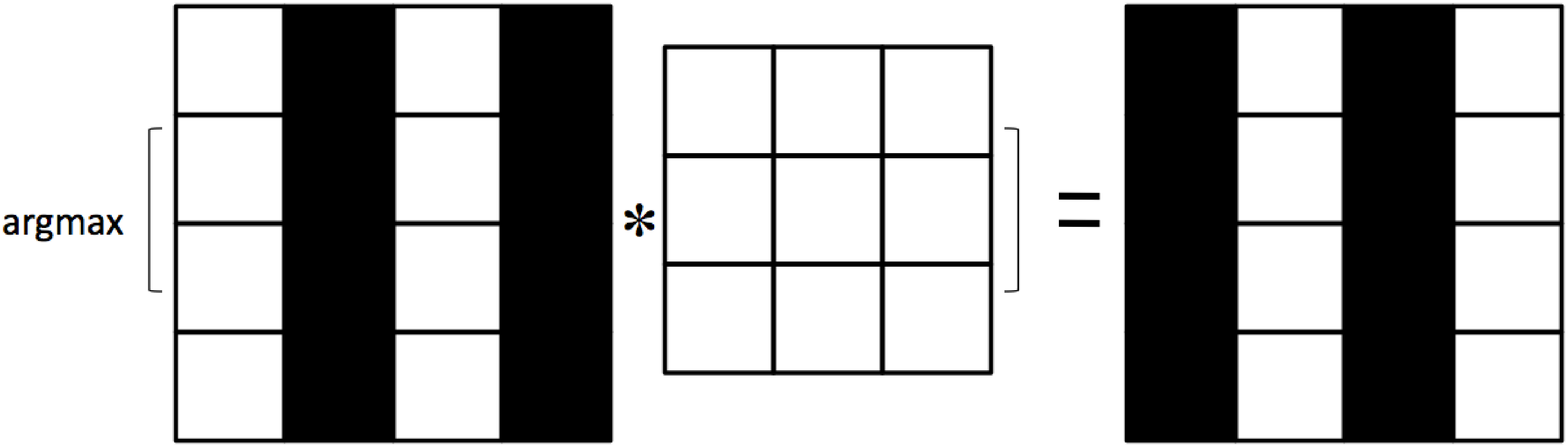}}
\linebreak{}
\subfloat[von Neumann's automaton]{\includegraphics[width=0.65\textwidth]{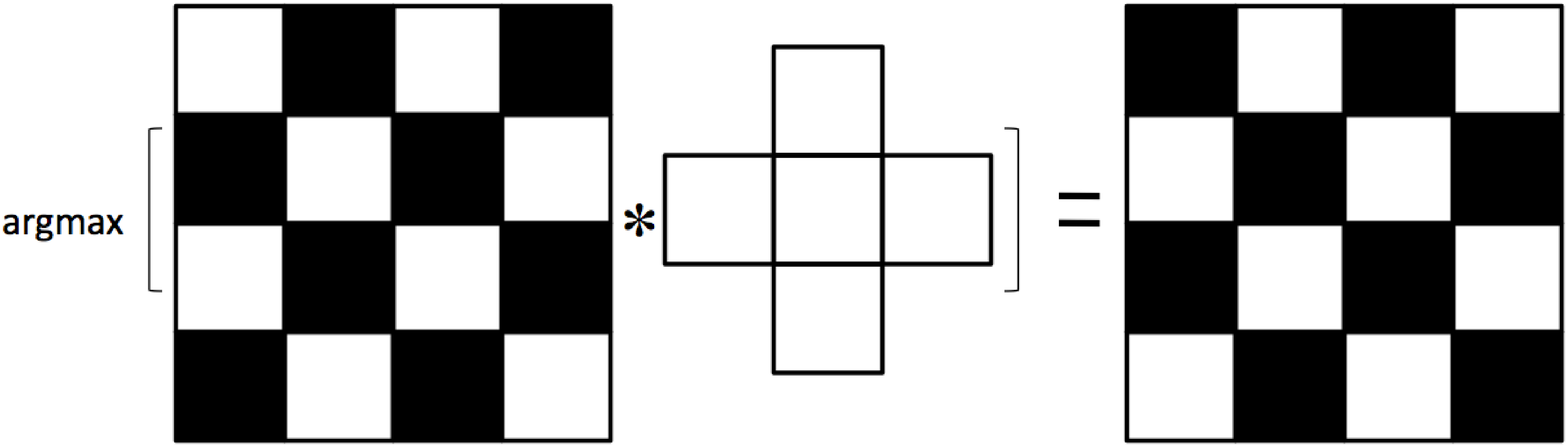}}
\caption{\label{fig:Oscillating states}\small An illustration of oscillating states produced by various automata: (a) The 3-tap box filter $\delta_{-1}+\delta_0+\delta_1$ over $\Omega=\Int_4$.  Using this $g$ in~\eqref{eq:ct} with $\psi_0=\delta_0+\delta_2$ results in the endless $2$-cycle $\delta_0+\delta_2\mapsto\delta_1+\delta_3\mapsto\delta_0+\delta_2$.  This is because at each iteration, each pixel's two neighbors will outvote him in deciding his label in the next iteration.  (b) Convergence is also an issue in two dimensions, as illustrated by Moore's automaton---a $3\times 3$ box filter---over $\Omega=\Int_4\times\Int_4$.  (c) Two-cycles persist in two dimensions even when the box filter is replaced by the smoother ``plus'' filter of von Neumann's automaton.  In all three cases, these filters are even  and so Theorem~\ref{thm:sym} ensures that the cycle length $K$ of~\eqref{eq:ct} is either $1$ or $2$.  However, none of them are positive semidefinite, as their DFTs attain negative values.  As such, the convergence guarantees of Theorem~\ref{thm:fp} do not hold.}
\end{figure*}

Of course, it is not difficult to find filters $g$ which do satisfy the hypotheses of Theorem~\ref{thm:suff}: one may simply let $g$ be the inverse DFT of any nonnegative even function.  More concrete examples, such as a discrete Gaussian over $\Int_N$, can be found using the following process.  Let $h:\Reals\rightarrow\Reals$ be an even Schwartz function whose Fourier transform is nonnegative; an example of such a function is a continuous Gaussian.  Let $g$ be the $N$-periodization of the integer samples of $h$, namely $g(n):=\sum_{n'=-\infty}^{\infty}h(n+Nn')$.  Then $g$ is even, and moreover, by the Poisson summation formula:
\begin{equation*}
(\rmF g)(n)
=\sum_{n'=0}^{N-1}g(n')\rme^{-\frac{2\pi\rmi nn'}{N}}
=\sum_{n'=0}^{N-1}\sum_{n''=-\infty}^{\infty}h(n'+Nn'')\rme^{-\frac{2\pi\rmi nn'}{N}}
=\sum_{k=-\infty}^{\infty}h(k)\rme^{-\frac{2\pi\rmi nk}{N}}
=\sum_{k=-\infty}^{\infty}\hat{h}(k+\tfrac nN)
\geq0.
\end{equation*}
In particular, if $g$ is chosen as a periodized version of the integer samples of any zero-mean Gaussian, then Theorem~\ref{thm:suff} gives that the AM algorithm~\eqref{eq:am} necessarily converges.  This construction method immediately generalizes to higher-dimensional settings where $D>1$.  It also generalizes to the noncircular convolution setting considered in Section~\ref{sec:BeyondSymmetry}.  There, we further restrict $h$ to be strictly positive, and let $g$ be the integer samples of $h$.  The positivity of $h$ implies $(\chi_{\Omega}* g)(n)>0$ for all $n\in\Omega$, implying $g\in\mathcal{G}(\Omega)$ as needed.  Moreover, $g$ is even and the Poisson summation formula gives that its Fourier series is nonnegative: $\hat{g}(x)=\sum_{k=-\infty}^{\infty}\hat{h}(k+x)\geq0$.  Any $g$ constructed in this manner satisfies the hypotheses of Corollary~\ref{cor:noncirculant}, implying the corresponding noncirculant AM~\eqref{eq:am,noncircular} necessarily converges.

\subsection{The rate of convergence of the AM algorithm}
\begin{figure*}
\centering{\includegraphics[width=0.35\textwidth]{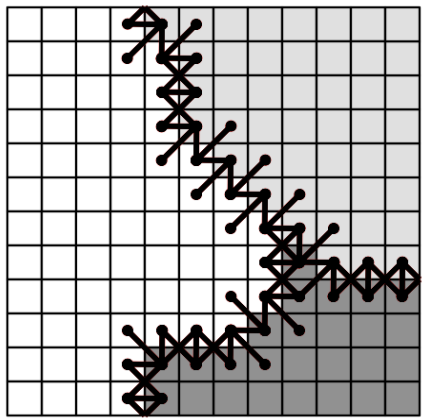}}
\caption{\label{fig:Stitches}\small An illustration of the zero-crossings in an image with $M$=3 masks.}
\end{figure*}
\begin{figure*}
\centering
\includegraphics[width=0.65\textwidth]{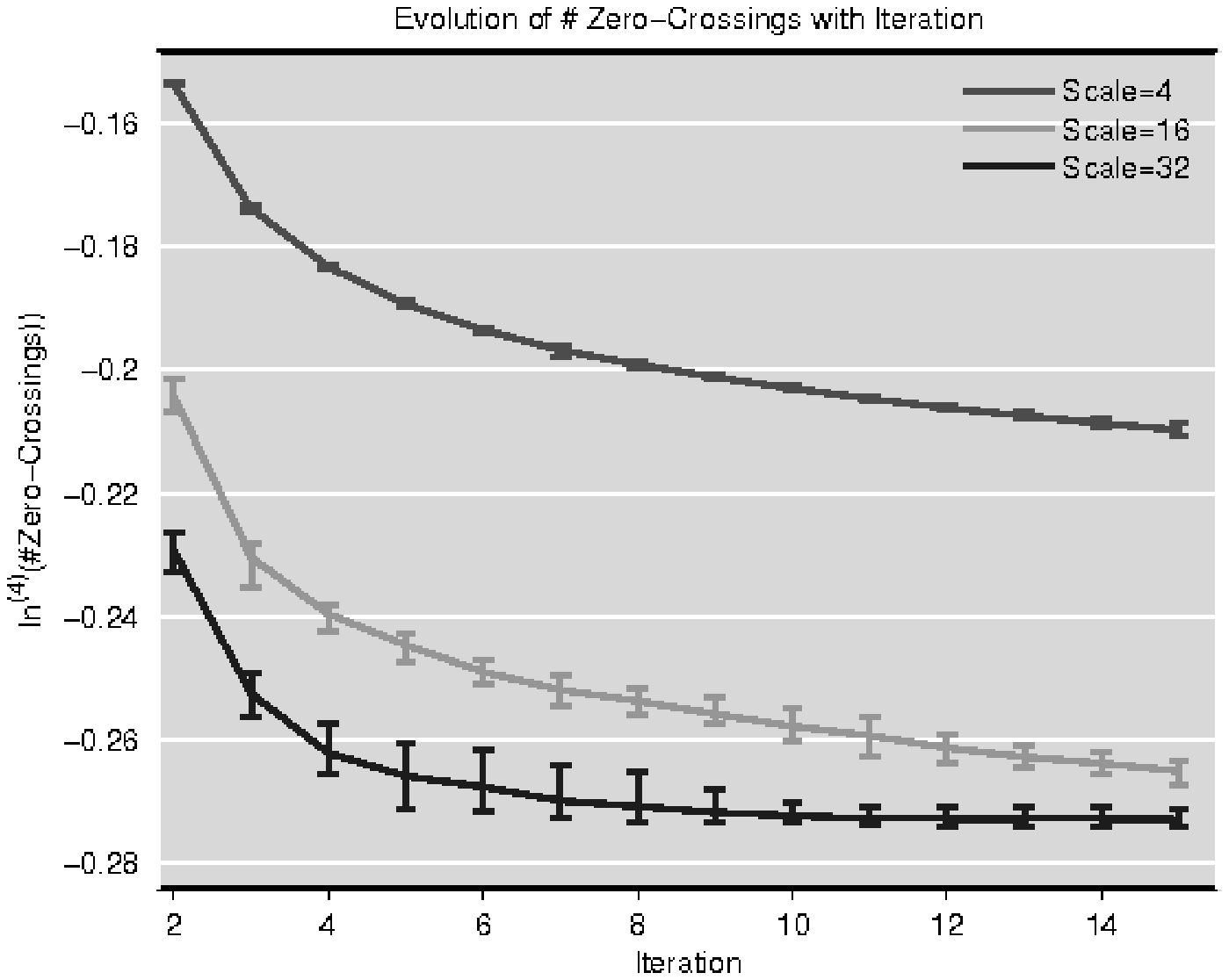}
\caption{{\small \label{fig:conv-diff-scales} The rate of decrease of the AM algorithm in terms of the number of boundary crossings.}}
\end{figure*}
Up to this point, we have focused on the question of whether or not the AM algorithm~\eqref{eq:am} converges.  Having settled that question to some degree, our focus now turns to another question of primary importance in real-world implementation: at what rate does AM converge?  Experimentation reveals that this rate highly depends on the configuration of the boundary between two distinctly labeled regions of $\Omega$.  This led us to postulate that the number of \textit{boundary crossings} (see Figure~\ref{fig:Stitches}) should monotonically decrease with each iteration.  Experimentation reveals that this number indeed often decreases extremely rapidly, regardless of the scale of $g$.  Figure~\ref{fig:conv-diff-scales} depicts such an experiment for the fluorescence microscope image shown in Figure~\ref{fig:AM}(a). Starting from a random initial configuration of 64 masks, we used a Gaussian filter under three different scales, with each plot depicting the evolution of 5 independently-initialized runs of the algorithm.  We emphasize the algorithm's fast rate of convergence: the vertical axis represents a nested four-fold application of the natural logarithm to the number of boundary crossings.  We leave a more rigorous investigation of the AM algorithm's rate of convergence for future work.

\section*{Acknowledgments}
We thank Prof.~Adam D.~Linstedt and Dr. Yusong Guo for providing the biological images which were the original inspiration for the AM algorithm and this work.
Fickus and Kova\v{c}evi\'{c} were jointly supported by NSF CCF 1017278.  Fickus received additional support from NSF DMS 1042701 and AFOSR F1ATA00183G003, F1ATA00083G004 and F1ATA0035J001. Kova\v{c}evi\'{c} also received  support from NIH R03-EB008870. The views expressed in this article are those of the authors and do not reflect the official policy or position of the United States Air Force, Department of Defense, or the U.S.~Government.

\bibliographystyle{plain}
\bibliography{am,bibl_jelena}
\end{document}